\documentclass{amsart}
\usepackage{amsthm}
\usepackage{amsmath}
\usepackage{amssymb}
\usepackage{graphicx}

\def\RR{\mathbf{R}}
\def\ZZ{\mathbf{Z}}

\def\TT{\mathbf{T}}

\newcommand\cF{\mathcal{F}}
\newcommand\cG{\mathcal{G}}

\newcommand\cR{\mathcal{R}}

\def\SL{{\mbox{\it SL}}}

\def\Id{{I}}
\def\Int{{\mbox{\rm Int }}}

\def\std{{\mbox{\it \scriptsize std}}}

\def\ra{{\rightarrow}}
\def\st{{\;|\;}}
\def\del{{\partial}}

\def\TFol{\mbox{\rm tFol}}
\def\Diff{\mbox{\rm Diff}}
\def\Fr{\mbox{\rm Fr}}

\def\tre{3}

\def\hsp{{\hspace{3mm}}}

\newcommand\cl[1]{{\overline{#1}}}
\newcommand\ch[1]{{\check{#1}}}

\theoremstyle{plain}
\newtheorem{thm}{Theorem}[section]
\newtheorem{prop}[thm]{Proposition}
\newtheorem{lemma}[thm]{Lemma}
\newtheorem{cor}[thm]{Corollary}

\newtheorem{question}[thm]{Question}

\theoremstyle{definition}
\newtheorem{dfn}[thm]{Definition}

\title[Homotopy classes of total foliations]{
 Homotopy classes of total foliations
 and bi-contact structures on three-manifolds}
\author{Masayuki Asaoka} 
\address{Department of Mathematics, Kyoto University}
\email{asaoka@math.kyoto-u.ac.jp}
\author{Emmanuel Dufraine}
\email{emmanuel.dufraine@gmail.com}
\author{Takeo Noda}
\address{Department of Mathematics, Toho University}
\email{noda@c.sci.toho-u.ac.jp}
\thanks{The first and third author were partially supported by JSPS PostDoctoral
Fellowships for Research Abroad.
They also are partially supported by Grant-in-Aid for Young
 Scientists (B) No. 19740085 and No. 19740026 from MEXT Japan.}

\begin{document}
\sloppy

\maketitle
\begin{abstract}
On every compact and orientable three-manifold,
 we construct total foliations
 (three codimension-one foliations that are transverse at every point).
This construction can be performed
 on any homotopy class of plane fields with vanishing Euler class.

As a corollary we obtain similar results on bi-contact structures.
\end{abstract}
\tableofcontents

\section{Introduction}

\subsection{Main results}

Let $M$ be an oriented closed three-dimensional manifold.
We call a triple $(\xi^i)_{i=1}^3$ of smooth transversely oriented
 plane fields on $M$ {\it a total plane field}
 if $\bigcap_{i=1}^3 \xi^i(p)=\{0\}$ for any $p$ in $M$.
If each $\xi^i$ is integrable, it is called {\it a total foliation}.
We say two total plane fields are {\it homotopic}
 if they are connected by a continuous path
 in the space of smooth oriented total plane fields.

A celebrated theorem due to Wood \cite{Wo} showed
 that any plane field on a closed three-dimensional manifold
 can be continuously deformed into a foliation {\em in its homotopy class}.
In other words,
 there is no homotopical obstruction to the integrability
 for the three-dimensional case.
The main subject of this paper is to solve the analogous problem
 for total foliations. That is,
\begin{thm}
\label{thm:total}
Any total plane field on a closed three-dimensional
 manifold is homotopic to a total foliation.
\end{thm}
In other words, there is no homotopical obstruction to 
 the integrability for total plane fields.

Let us remark that three-dimensional closed manifolds
 have their Euler characteristic equal to zero,
 which implies the existence of transversely oriented plane fields.
Similarly, three-dimensional closed manifolds have
 vanishing second Stiefel-Whitney class,
 which implies the existence of total plane fields.

Hardorp \cite{Ha}, showed that
 any three-dimensional oriented  closed manifold admits a total foliation.
However, his construction does not allow to keep track of the homotopy class of the constructed object.

Tamura and Sato \cite{TS},
 gave examples of foliations on three-dimensional manifold
 which admit a transverse plane field but
 no transverse foliation.
It implies that there exists an obstruction to
 deform a total plane field into a total foliation
 {\it if we fix one of the plane fields as a given foliation}.

Mitsumatsu \cite[Problem 5.2.7]{Mi2},
 asked which homotopy classes of plane fields can be
 realized as a transverse pair of codimension-one foliations.
His question is important from the viewpoint of
 bi-contact structures, which we consider in the next paragraph.
The theory of characteristic classes tells that
 a plane field is contained in a total plane field
 if and only if its Euler class vanishes.
Theorem \ref{thm:total} answers Mitsumatsu's question immediately.
\begin{cor}
\label{cor:total}
An oriented plane field
 on an oriented closed three-dimensional manifold
 is homotopic to a foliation which is contained in a total foliation
 if and only if its Euler class vanish.
\end{cor}
\medskip

We call a pair of mutually transverse 
 positive and negative contact structures
 {\it a bi-contact structure}.
Mitsumatsu \cite{Mi}, and Eliashberg and Thurston \cite{ET}
 showed that bi-contact structures
 naturally correspond to {\it a projectively Anosov flow},
 which exhibits partially-hyperbolic behavior on the whole manifold.

Related to the question above,
 Mitsumatsu asked which homotopy class of plane field
 can be realized by contact structures in a bi-contact structure.
In \cite[Theorem 2.4.1]{ET},
 Eliashberg and Thurston showed that
 any foliation except the product foliation
 $\{S^2 \times \{p\}\}_{p \in S^1}$ on $S^2 \times S^1$
 can be $C^0$-approximated by positive or negative contact structure.
It is easy to see that any mutually transverse plane fields
 are homotopic to each other
 and that the product foliation on $S^2 \times S^1$
 does not admit a transverse foliation.
Hence, the following is an immediate consequence
 of Eliashberg-Thurston's theorem and Corollary~\ref{cor:total}.
\begin{cor}
\label{cor:bi-contact} 
On any oriented closed three-dimensional manifold,
 any oriented plane field with Euler class zero
 is homotopic to positive and negative contact structures
 which form a bi-contact structure.
\end{cor}

Among the realization problems of bi-contact structures,
 the following is quite natural.
\begin{question}
Let $\xi$ and $\eta$ be positive and negative contact structures
 on an oriented three-dimensional manifold $M$.
Suppose that they are contained in the same homotopy class
 of plane fields with vanishing Euler class.
Can we isotope $\xi$ and $\eta$
 so that $(\xi,\eta)$ is a bi-contact structure?
\end{question}
We give an answer for overtwisted contact structures.
\begin{thm}
\label{thm:bi-contact}
Let $\xi$ and $\eta$ be positive and negative overtwisted contact
 structures contained in the same homotopy class of plane fields
 and with Euler class zero.
Then, we can isotope $\xi$ and $\eta$
 so that $(\xi,\eta)$ is a bi-contact structure.
\end{thm}
The answer for tight contact structures is still unknown.

\subsection{Outline of Proofs}
Proof of Theorem \ref{thm:total} is obtained after performing
 a sequence of surgeries and gluings
 along so-called {\it $\cR$-components},
 which are solid tori equipped with a `simple' total foliation.

Section~\ref{chap:formula} is devoted to the study of the effect
 of a surgery on the homotopy class of a total foliation.
In subsection~\ref{sec:spin},
 we review two invariants of total plane fields
 that determine its homotopy class completely
 -- {\it the spin structure} and {\it the difference of Hopf degree}.
In subsection~\ref{sec:Reeb},
 we define $\cR$-components of total foliations
 and gluing of two total foliations along the boundaries
 of $\cR$-components.
In subsections~\ref{sec:surgery} and~\ref{sec:gluing},
 we define a surgery of a total foliation along
 an $\cR$-component and give a surgery formula.

Section \ref{chap:construction} is the main part of our construction
 of a total foliation in any given homotopy class.
It is done by a modification of Hardorp's construction in \cite{Ha}.
The main new feature in our construction
 is a control of the framing of $\cR$ components
 by insertion of `plugs' (Lemma \ref{lemma:framing change}).
Insertion of plugs of another type also
 enables us to control the difference of Hopf degree
 (Lemma \ref{lemma:change Hopf}).
In order to obtain such plugs,
 we need to construct total foliations on the three-dimensional
 sphere $S^3$ such that the cores of $\cR$-components 
 form special framed links.
Hardorp's construction is insufficient to our purpose 
 since the framing is a very large positive number
 and it is difficult to control.
In the first step of our construction,
 there are two differences from his construction :
\begin{enumerate}
 \item our construction is performed on a non-trivial $\TT^2$-bundle
 over the circle while Hardorp's was on $\TT^3$;
 \item foliations in our $\cR$-components may rotate several times
 in some sense while they did not in Hardorp's.
\end{enumerate}
These differences leads to a simpler construction
 in the succeeding steps:
 we can avoid dealing with a finite covering of
 a total foliation on the Poincar\'e sphere 
 and with a branched double covering along the unknot.
As a consequence, we can obtain an explicit description
 of the framings of $\cR$-components in terms of diagrams of braids,
 see Proposition \ref{prop:any braid}.

In subsection \ref{sec:W}, we give a construction of total foliations
 on $\TT^2 \times [0,1]$.
In subsection \ref{sec:trefoil},
 we describe the framings of $\cR$-components of a total foliation
 that is given by gluing two boundary components of $\TT^2 \times [0,1]$.
In subsection \ref{sec:Hardorp},
 we control the framings of $\cR$-components
 and show a generalized version of Hardorp's theorem,
 {\it i.e.}, the existence of a total foliation with any given
 spin structure.
The control is done by successive replacements of an $\cR$-component
 with a totally foliated solid torus
 which contains a twisted $\cR$-component (`insertion of plugs').
In subsection \ref{sec:proof}, we give a control of the Hopf degree.
In fact, we construct a total foliation on $S^3$
 that admits unknotted $\cR$-components
 with $(+1)$- and $(-1)$-framings
 and that has the required difference of Hopf degree
 with the positive total Reeb foliation.
By gluing it with a total foliation
 that has the required spin structure,
 we obtain a total foliation
 in any given homotopy class of total plane fields.

Section \ref{chap:contact} is devoted to
 the proof of Corollary \ref{cor:bi-contact}.
We show that if a total foliation admits an unknotted $\cR$-component
 with $(+1)$-framing then any positive contact structure
 that is sufficiently close to one of the foliations
 violates the Thurston-Bennequin inequality and therefore is overtwisted.
Once it is shown, the corollary is an easy consequence of
 Eliashberg's classification of overtwisted contact structures
 in \cite{El}.

\subsection{Acknowledgements}
This paper was prepared while the first and third authors stayed at Unit\'e
 de Math\'ematiques Pures et Appliqu\'ees, \'Ecole Normale
 Sup\'erieure de Lyon and it started when the second author was at Institut Fourier, Grenoble. They thank the members of those institutions, especially
 Professor \'Etienne Ghys for his warm hospitality.
The authors are also grateful to an anonymous referee
 for many suggestions to improve the readability of the paper.

\section{Gluing and surgery of total foliations}\label{chap:formula}

\subsection{Homotopy classes of plane fields}
\label{sec:spin}
In the rest of the paper, all manifolds and foliations are of class $C^\infty$ and all plane fields and foliations are transversely oriented.

Fix an $n$-dimensional manifold $X$ equipped with a Riemannian metric.
Let $\Fr(X)$ be the set of orthonormal frame of $TX$.
It admits a natural topology as a subset of
 the set of $n$-tuples of vector fields on $X$.

When $M$ is a three-dimensional manifold,
 by taking the unit normal vectors of a total plane field,
 and by applying the Gram-Schmit orthogonalization to it,
 we can define a continuous map from the set of total plane fields
 to $\Fr(M)$.
It is easy to see that it induces a bijection between homotopy classes.
So, we consider $\Fr(M)$ instead of the set of total plane fields
 in this subsection.

First, we review some basic facts on spin structures.
We denote by $SO(n)$ the group of special orthogonal matrices
 of size $n$.
Let $X$ be an $n$-dimensional manifold with $n \geq 3$.
We fix a triangulation of $X$
 and let $X_i$ be the $i$-skeleton of $X$ for $0 \leq i \leq n$.
By $\Fr(X_i)$, we denote the set of orthonormal frames of $TX|_{X_i}$.
A {\it spin structure} is a homotopy class of $\Fr(X_1)$ of
 which each representative can be extended to an element of $\Fr(X_2)$.
In particular, a frame $\ch{e}$ in $\Fr(X)$
 induces a spin structure on $X$ in a natural way.
We call it {\it the spin structure} given by the frame $\ch{e}$.
Our definition is different from the standard one that is given by
 a double covering of a natural principal $SO(n)$-bundle,
 but it is known they are equivalent if $n \geq 3$, see \cite{Mil}.

A manifold $X$ equipped with spin structure $s$ is called
 {\it a spin manifold}.
If $X$ has a boundary $\del X$, then $s$ induces
 a spin structure $s'$ on $\del X$.
We call the spin manifold $(\del X,s')$
 {\it the spin boundary} of $(X,s)$.

Now, we focus our attention on spin structures on
 three or four-dimensional manifolds.
We call a four-dimensional manifold $X$ {\it a $2$-handlebody} if
 it is obtained by attaching four-dimensional $2$-handles
 to the $4$-ball $B^4$ along a framed link $L$ in $S^3 =\del B^4$.
We say a $2$-handlebody $X$ is {\it even}
 if the framing of each component of $L$ is even.
See the first paragraph of Subsection \ref{sec:surgery}
 for the definition of framing of knots.

\begin{prop}
\label{prop:2-handlebody}
Any even $2$-handlebody admits a unique spin structure.
Any closed spin three-dimensional manifold
 is a spin boundary of a spin $2$-handlebody.
\end{prop}
\begin{proof}
See Section 5.6 and 5.7 of \cite{GS}.
\end{proof}

Let $M$ be a three-dimensional closed manifold.
We denote by $C(M,SO(3))$ the set of continuous maps from $M$ to $SO(3)$.
The space $\Fr(M)$ of frames admits a natural action of $C(M,SO(3))$
 given by $(F \cdot (e^i)_{i=1}^3)(p)=(F(p) \cdot e^i(p))_{i=1}^3$
 for $F \in C(M,SO(3))$.
We define a map $\Phi:\Fr(M) \times \Fr(M) \ra C(M,SO(3))$
 by $\ch{e}=\Phi(\ch{e},\ch{e}_0)\cdot \ch{e}_0$
 for $(\ch{e},\ch{e}_0) \in \Fr(M)^2$.
It is easy to check that $\Phi(\cdot,\ch{e}_0)$ is a bijective map
 between $\Fr(M)$ and $C(M,SO(3))$.

We denote the field $\ZZ/2\ZZ$ by $\ZZ_2$.
Recall the fundamental group $\pi_1(SO(n))$ of $SO(n)$
 is isomorphic to $\ZZ_2$ if $n \geq 3$.
Let $Spin(n)$ be the universal covering group of $SO(n)$.
\begin{dfn}
For $\ch{e},\ch{e_0} \in \Fr(M)$,
 we define $s(\ch{e},\ch{e_0}) \in H^1(M,\ZZ_2)$ by
\begin{displaymath}
 s(\ch{e},\ch{e_0})([\gamma]_{H_1})
 =[\Phi(\ch{e},\ch{e_0}) \circ \gamma]_{\pi_1}
 \in \pi_1(SO(n)) \simeq \ZZ_2
\end{displaymath}
 for any continuous loop $\gamma$ in $M$.
We call the above cohomology class
 {\it the difference of spin structures}
 of $\ch{e}$ and $\ch{e}_0$.
\end{dfn}
It is easy to see that
 $s(\ch{e},\ch{e_0})$ is well-defined and
 is determined by the homotopy classes of $\ch{e}$ and $\ch{e}_0$.
We can see that
 $s(\ch{e},\ch{e}_0)=0$ if and only if
 the restrictions of $\ch{e}$ and $\ch{e}_0$ to a fixed $1$-skeleton
 are homotopic.
In particular, $s(\ch{e},\ch{e}_0)=0$ if and only if
 two frames $\ch{e}$ and $\ch{e}_0$ give the same spin structure.

\begin{lemma}
\label{lemma:spin 1}
If two given frames $\ch{e}, \ch{e}_0 \in \Fr(M)$
 satisfy $s(\ch{e},\ch{e}_0)=0$,
 then the map $\Phi(\ch{e},\ch{e}_0)$ admits a lift
 $\tilde{\Phi}(\ch{e},\ch{e}_0):M \ra Spin(3)$.
\end{lemma}
\begin{proof}
The map $\Phi(\ch{e},\ch{e}_0)$ induces a trivial map between
 the fundamental groups.
Hence, it admits a lift to $Spin(3)$.
\end{proof}

\begin{dfn}
\label{def:Hopf}
When two frames $\ch{e}$ and $\ch{e}_0$ of $M$
 give the same spin structure,
 we define {\it the difference of Hopf degree} $H(\ch{e},\ch{e}_0)$ 
 by the mapping degree of $\tilde{\Phi}(\ch{e},\ch{e}_0)$.
\end{dfn}
Remark that
 $H((e^i)_{i=1}^3,(e_0^i)_{i=1}^3)$ coincides with
 {\it the difference of Hopf degree of non-singular vector fields}
 $e^i$ and $e_0^i$ for any $i=1,2,3$, which is defined in \cite{Du}.
It is easy to see that the formulae
\begin{align}
\label{eqn:Hopf 0}
H(\ch{e}_2,\ch{e}_1)
 &= H(-\ch{e}_1,-\ch{e}_2) =  -H(\ch{e}_1,\ch{e}_2),\\
\label{eqn:Hopf} 
H(\ch{e}_1,\ch{e}_3)
 &= H(\ch{e}_1,\ch{e}_2) + H(\ch{e}_2,\ch{e}_3)
\end{align}
 hold if $\ch{e}_1,\ch{e}_2,\ch{e}_3 \in \Fr(M)$
 give the same spin structure,
 where $-\ch{e}=(-e^i)_{i=1}^3$ for $\ch{e}=(e^i)_{i=1}^3$.

\begin{prop}
\label{prop:framing}
Two frames $\ch{e},\ch{e}_0 \in \Fr(M)$ are homotopic to each other
 if and only if they give the same spin structure and
 satisfy $H(\ch{e},\ch{e}_0)=0$.
\end{prop}
\begin{proof}
It is trivial that the former implies the latter.

Suppose the latter holds for $\ch{e},\ch{e}_0 \in \Fr(M)$.
Then, we have $s(\ch{e},\ch{e}_0)=0$ and $H(\ch{e},\ch{e}_0)=0$.
Fix a structure of a CW complex on $M$ with a unique $3$-cell.
Let $M_2$ be the $2$-skeleton of $M$.
Since $Spin(3)$ is homeomorphic to $S^3$,
 the lift $\tilde{\Phi}(\ch{e},\ch{e}_0)$
 of $\Phi(\ch{e},\ch{e}_0)$ is homotopic to
 a map $F$ such that $F|_{M_2}$ is a constant map.
Since the quotient space $M/M_2$ also is homeomorphic to $S^3$,
 the assumption $H(\ch{e},\ch{e}_0)=0$ implies
 that $F$ is homotopic to a constant map.
Therefore, $\ch{e}$ is homotopic to $\ch{e}_0$.
\end{proof}

\subsection{$\cR$-components and gluing of total foliations}
\label{sec:Reeb}
In the rest of the paper,
 we identify the circle $S^1$ with $\RR/\ZZ$,
 and the two-dimensional torus $\TT^2$ with $(\RR/\ZZ)^2$.
The sum $a+b$ is
 well-defined for $a \in S^1$ and $b \in S^1$ or $\RR$.
For $a \in S^1$ and $\epsilon_1,\epsilon_2 \in \RR$,
 we denote the subset $\{a+t \in S^1
 \st t \in [\epsilon_1,\epsilon_2]\}$
 by $[a+\epsilon_1,a+\epsilon_2]$.
We will abuse the identification of the number $t \in [0,1]$
 and $t + \ZZ \in S^1$ when the meaning is clear.

Put $D^2(r)=\{(x,y) \in \RR^2 \st x^2+y^2 \leq r^2\}$ for $r \geq 0$
 and $D^2=D^2(1)$.
We denote $[0,1] \times \TT^2$ by $W$, $S^1 \times D^2$ by $Z$,
 and the origin of $\RR^2$ by $O$.
We also denote by $]a,b[$ the open interval $\{x \in \RR \st a <x <b\}$.

For a foliation $\cF$ on a manifold $X$ and a point $p$ of $X$,
 let $\cF(p)$ denote the leaf containing $p$.
For a diffeomorphism $F$ from $X$ to another manifold $X'$,
 let $F(\cF)$ denote a foliation on $X'$ such that the leaf
 containing $F(p)$ is $F(\cF(p))$.
For a pair $(\cF^1,\cF^2)$ of mutually transverse codimension-one
 foliations on a three-dimensional manifold $M$,
 let $\cF^1 \cap \cF^2$ be the one-dimensional foliation
 $\{\cF^1(p) \cap \cF^2(p)\}_{p \in M}$.

\begin{dfn}
Let $M$ be a three-dimensional manifold.
We say a subset $R$ of $M$ is {\it a thick Reeb component}
 of a foliation $\cF$ if $R$ contains a Reeb component $R'$
 and $\cF|_{\cl{R \setminus R'}}$ is diffeomorphic to a product foliation
 $\{t \times \TT^2\}_{t \in [0,1]}$ on $W$.
\end{dfn}

Let $(t,x,y)$ be the standard coordinate system of
 $S^1 \times \RR^2$.
Take a smooth odd function $\chi_R$ on $\RR$ so that
 $0<\chi_R(x)<1$ if $x \in ]1/2,3/2[$
 and $\chi_R(x)=0$ otherwise.
Let $\hat{\cR}^1$ and $\hat{\cR}^2$ be the foliations
 on $S^1 \times \RR^2$
 that are generated by the kernel of
 $dy-\chi_R(y)dt$ and $dx-\chi_R(x)dt$, respectively.

We denote by $\cR^i$ the restriction of $\hat{\cR^i}$ on $Z$ for $i=1,2$.
We can take a foliation $\cR^3$ on $Z$
 so that it is a thick Reeb component
 and $(\cR^i)_{i=1}^3$ is a total foliation.
See Figure \ref{fig:Reeb}.
\begin{figure}[ht]
\begin{center}
\includegraphics[scale=1.0]{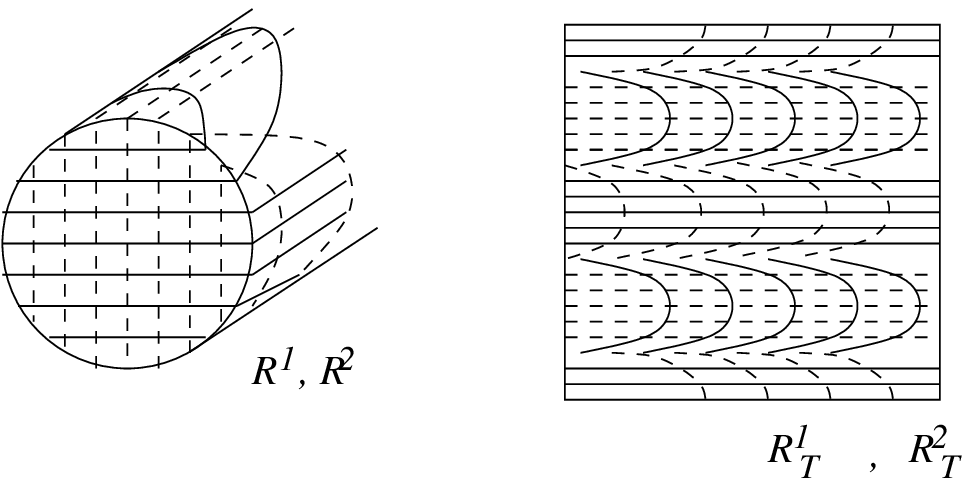}
\end{center}
\caption{Foliations $\cR^1$, $\cR^2$, $\cR_T^1$, and $\cR_T^2$}
\label{fig:Reeb}
\end{figure}

\begin{dfn}
\label{dfn:R-component}
Let $(\cF_i)_{i=1}^3$ be a total foliation
 on a three-dimensional manifold $M$.
We call a subset $R$ of $M$ {\it an $\cR$-component} of $(\cF^i)_{i=1}^3$
 if there exists a diffeomorphism $\psi:Z \ra R$ such that
 $\psi(\cR^i)=\cF^i|_R$ for $i=1,2,3$
 and the restriction of $\cF^3$ on a neighborhood of $\del R$
 is diffeomorphic to $\{t \times \TT^2\}_{t \in [0,1]}$ on $W$.
The diffeomorphism $\psi$ is called {\it a canonical coordinate} of $R$.
The curve $C(R)=\psi(S^1 \times 0)$ admits a natural orientation
 induced from $\psi$ and we call it {\it the core} of $R$.
\end{dfn}
Remark that the isotopy class of $C(R)$ is uniquely determined
 as an oriented knot in $M$.

Let $\varphi_\cR:\TT^2 \ra S^1 \times \del D^2$ be the map 
 given by $\varphi_\cR(x,y)=(x,\cos(2\pi y),\sin(2\pi y))$.
We define foliations $\cR_T^1$ and $\cR_T^2$ on $\TT^2$
 so that $\varphi_\cR(\cR_T^i)$ is the restriction of $\cR^i$
 on $\del Z$ for each $i=1,2$.
We use the following lemma in Section \ref{chap:contact}.
\begin{lemma}
\label{lemma:perturbation}
If a smooth line field $\xi$ on $\TT^2$ is sufficiently $C^0$-close
 to $T\cR_T^1$, then there exists a closed curve which is
 tangent to $\xi$ and homotopic to the curve $S^1 \times y_0$,
 where $y_0$ be the point of $S^1$ represented by $0$.
\end{lemma}
\begin{proof}
Put $A=S^1 \times [y_0-1/4,y_0+1/4]$.
If a smooth line field $\xi$ on $\TT^2$ is sufficuently $C^0$-close
 to $T\cR_T^1$, then it is isotopic to $\del A$
 and admits an orientation which directs inward at $\del A$.
By the Poincar\'e-Bendixon theorem,
 there exists a closed curve in $A$
 which is tangent to $\xi$ and isotopic to $S^1 \times y_0$.
\end{proof}

Let $a_\cR$ be the integral homology class in $H_1(\TT^2,\ZZ)$
 represented by a map $x \mapsto (x,0)$.
Remark that each closed leaf of $\cR_T^1$
 is the image of a curve which represents $a_\cR$.
\begin{dfn}
\label{dfn:R-boundary}
Let $(\cF^i)_{i=1}^3$ be a total foliation on a manifold $M$.
We call a boundary component $T$ of $M$ {\it an $\cR$-boundary}
 if there exists a diffeomorphism $\psi_T:\TT^2 \ra T$
 such that $\psi_T(\cR_T^i)$ is the restriction of $\cF^i$ to $T$
 for $i=1,2$,
 and $\cF^3$ is diffeomorphic to
 the product foliation $\{t \times \TT^2\}_{t \in [0,1]}$
 on a neighborhood of $T$.
For an $\cR$-boundary component $T$,
 we define $a_\cR(T) \in H_1(T,\ZZ)$ by $a_\cR(T)=(\psi_T)_*(a_\cR)$.
\end{dfn}
Remark that if $R$ is an $\cR$-component of a total foliation
 on a manifold $M$, then $\del R$ is an $\cR$-boundary
 of both $R$ and $\cl{M \setminus R}$.

We define cut and paste operations of total foliations
 with $\cR$-boundary by following the idea described in \cite{Ha}.
First, we show that the pair $(\cF^1|_T,\cF^2|_T)$ of foliations
 of an $\cR$-boundary of a total foliation $(\cF^i)_{i=1}^3$
 is determined by $a_\cR(T)$ up to isotopy.
\begin{lemma}
\label{lemma:isotopy} 
Let $F$ be a diffeomorphism of $\TT^2$
 such that $F_*(a_\cR)=a_\cR$.
Then, there exists a diffeomorphism $G$
 which is isotopic to the identity
 and satisfies $G(\cR_T^i)=F(\cR_T^i)$ for $i=1,2$.
\end{lemma}
\begin{proof}
Let $\tau_y$ be the diffeomorphism of $\TT^2$
 such that $\tau_y(x,y)=(x,-y)$.
Then, $\tau_y(\cR_T^i)=\cR_T^i$ for $i=1,2$
 and $(\tau_y)_*(a_\cR)=a_\cR$.
Hence, we may assume that $F$ is orientation-preserving
 by replacing $F$ with $F \circ \tau_y$ if it is necessary.

Fix an integer $k$.
Let $\tilde{h}$ be a smooth function on $\RR \times [0,1]$
 such that $\tilde{h}(y+n,t)=\tilde{h}(y,t)+kn$
 for any $(y,t) \in \RR \times [0,1]$ and $n \in \ZZ$, and
\begin{equation*}
\tilde{h}(y,t)=\left\{
\begin{array}{ll}
kn & \mbox{ if } y \in [n,n+(1/32)] \\
k(n+t) & \mbox{ if } y \in [n+(1/16),n+(1/4)] \\
k(n+1) & \mbox{ if } y \in [n+(9/32),n+1]
\end{array}
\right. 
\end{equation*}
 for any $n \in \ZZ$ and $t \in [0,1]$.
See Figure \ref{fig:isotopy}.
The function $\tilde{h}$ induces a map $h:S^1 \times [0,1] \ra S^1$.
Remark that $h(\cdot,t):S^1 \ra S^1$ is a map of degree $k$.

For $t \in [0,1]$,
 we define a diffeomorphism $F_{k,t}$ of $\TT^2$ by
$F_{k,t}(x,y) =\left(x+h(y,t),y \right)$.
Since $F$ is orientation-preserving and $F_*(a_\cR)=a_\cR$,
 $F$ is isotopic to $F_{k,1}$ for some $k \in \ZZ$.
Hence, it is sufficient to show that
 there exists a diffeomorphism $G$ of $\TT^2$
 which is isotopic to the identity
 and satisfies $G(\cR_T^i)=F_{k,1}(\cR_T^i)$ for $i=1,2$.
\begin{figure}[ht]
\begin{center}
\includegraphics[scale=1.0]{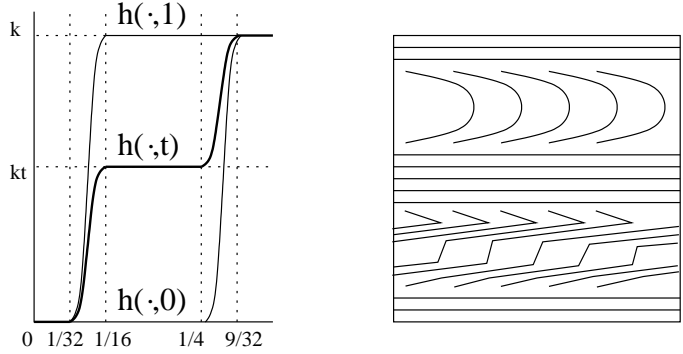}
\end{center}
\caption{The map $h(\cdot,t)$
 and the foliation $F_{k,1}(\cR^1)$ for $k=1$}
\label{fig:isotopy}
\end{figure}

Since $\cR_T^1(x,y_1)=S^1 \times y_1$ for $y_1 \in [0,1/16]$
 and $\cR_T^2(x,y_2)=S^1 \times y_2$ for $y_2 \in [1/4,5/16]$,
 we have $F_{k,0}(\cR_T^2)=\cR_T^2$
 and $F_{k,1}(\cR_T^1)=\cR_T^1$.
The foliations $\cR_T^1$ and $\cR_T^2$
 are invariant under the translation $(x,y) \mapsto (x+t,y)$.
It implies that $F_{k,t}(\cR_T^2)$ is transverse to $\cR_T^1$
 for any $t \in [0,1]$.
We define an isotopy $\{G_t\}_{t \in [0,1]}$
 by $G_t(x,y) \in \cR_T^1(x,y) \cap F_{k,t}(\cR_T^2(x,y))$.
Then, the map $G_0$ is the identity,
 $G_1(\cR_T^1)=\cR_T^1=F_{k,1}(\cR_T^1)$,
 and $G_1(\cR_T^2)=F_{k,1}(\cR_T^2)$.
\end{proof}

\begin{prop}
\label{prop:gluing} 
For $k=1,2$,
 let $M_k$ be a three-dimensional manifold with a toral boundary $T_k$
 and $(\cF_k^i)_{i=1}^3$ a total foliation of $M_k$
 such that $T_k$ is an $\cR$-boundary.
Suppose that a diffeomorphism $\psi:T_1 \ra T_2$
 satisfies $\psi_*(a_\cR(T_1))=a_\cR(T_2)$.
Then, there exists a total foliation
 $(\cF^i)_{i=1}^3$ on
 $M_1 \cup_\psi M_2 =M_1 \cup M_2/[p \sim \psi(p)]$
 and diffeomorphisms $F_1:M_1 \ra M_1$ and $F_2:M_2 \ra M_2$ such that
 $F_k$ is isotopic to the identity
 and $F_k(\cF_k^i)=\cF^i|_{M_k}$ for any $i=1,2,3$ and $k=1,2$.
\end{prop}
\begin{proof}
By Lemma \ref{lemma:isotopy}, we can isotope $(\cF_2^i)_{i=1}^3$
 so that it is compatible with $(\cF_1^i)_{i=1}^3$
 on a neighborhood of $T_1=T_2$ in $M_1 \cup_\psi M_2$.
\end{proof}

\subsection{Knotted $\cR$-components and surgery}
\label{sec:surgery}

Let $M$ be an oriented three-dimensional manifold.
For a smooth link $L$ in $M$, let $\Fr(L;M)$ be the set of 
 vector fields $v:L \ra TM$ on $L$
 satisfying $v(p) \not\in T_p L$ for any $p \in L$.
{\it A framing} of $L$ is a connected component of $\Fr(L;M)$.
An oriented knot $K$ is null-homologous if and only if
 it admits a {\it Seifert surface} $S$, that is,
 an oriented embedded surface with $\del S=K$.
\begin{dfn}
\label{dfn:framing}
Suppose an oriented knot $K$ admits a Seifert surface $S$.
We call an orientation preserving embedding $\psi:S^1 \times D^2 \ra M$
 {\it an $n$-framed tubular coordinate} of $K$
 if the restriction of $\psi$ to $S^1 \times \{(0,0)\}$
 is an orientation preserving diffeomorphism onto $K$
 and the algebraic intersection number
 of $S$ and $\psi(\{S^1 \times \{(1,0)\}\}$ is $n$.
The framing represented by a vector field $v \in \Fr(K;M)$
 tangent to $\psi(\{S^1 \times [-1,1] \times \{0\}\})$
 is called {\it the $n$-framing} of $K$.
\end{dfn}
It is known that the $n$-framing of $K$
 does not depend on the choices of $S$ and $\psi$.

If a link $L$ is tangent to leaves of a foliation $\cF$,
 then a vector field $v_L$ on $L$ with $v_L(p) \in T_p\cF \setminus T_p L$
 gives a framing of $L$.
We call it {\it the framing given by $\cF$}.
We say an $\cR$-component $R$ of a total foliation $(\cF^i)_{i=1}^3$
 on $M$ is {\it null-homotopic} if the core $C(R)$ is null-homotopic.
In addition, if $\cF^1$ gives the $n$-framing of $C(R)$,
 we say that $R$ is an $n$-framed null-homotopic $\cR$-component.
A knot is called {\it unknotted} if it bounds an embedded disk.
We say an $\cR$-component of a total foliation on $M$
 is {\it unknotted} if the core is unknotted.

Suppose that a total foliation $(\cF^i)_{i=1}^3$ on $M$
 admits an $\cR$-component $R$.
Let $\mu(R) \in H_1(\del R,\ZZ)$ be the homology class
 represented by a meridian of $R$.
Up to isotopy, there exists a unique diffeomorphism $F$ on $\del R$
 such that $F_*(a_\cR(\del R))=a_\cR(\del R)$
 and $F_*(\mu(R))=\mu(R)+a_\cR(\del R)$.
We call $M_R=(\cl{M \setminus R} \cup R)/F(p)  \sim p$
 the manifold obtained by {\it the standard surgery} along $R$.
By Proposition \ref{prop:gluing},
 total foliations $(\cF^i|_{\cl{M\setminus R}})_{i=1}^3$
 and $(\cF^i|_R)_{i=1}^3$
 induce a total foliation $(\cF_F^i)_{i=1}^3$.
We call $(\cF_F^i)_{i=1}^3$ the total foliation
 obtained by {\it the standard surgery} along $R$.
In \cite[p.22--24]{Ha}, one can see
 another surgery along an $\cR$-component,
 which essentially yields the same foliation.

\begin{lemma}
\label{lemma:framing of surgery}
If $R$ is null-homotopic and $k$-framed, then the above $M_R$ is
 a manifold obtained by a Dehn surgery along $C(R)$
 with framing coefficient $k+1$.
\end{lemma}
\begin{proof}
Since $R$ is $k$-framed,
 $\lambda(R)=a_\cR(R)-k\mu(R)$ is represented by the longitude of $C(R)$
 corresponding to the $0$-framing.
The condition $F_*(\mu(R))=\lambda(R)+(k+1)\mu(R)$
 implies that the coefficient of the Dehn surgery is $k+1$.
\end{proof}

Let $(\cF^i)_{i=1}^3$ be a total foliation on $S^3$
 and $R_1,\cdots,R_k$ be
 its $\cR$-components with the $n_1,\cdots,n_k$-framings.
Lemma \ref{lemma:framing of surgery} implies that
 the manifold obtained by the standard surgery
 along $\cR$-components $R_1,\cdots,R_k$
 is the boundary of the four-dimensional $2$-handlebody
 $X$ of which Kirby diagram is $\bigcup_{j=1}^k C(R)$
 with the $(n_j+1)$-framing on each $C(R_j)$.

As we saw in Subsection \ref{sec:spin},
 each total plane field on $M$ defines a spin structure on $M$.
For a total foliation $(\cF^i)_{i=1}^3$,
 we say a spin structure on $M$ is {\it given by} $(\cF^i)_{i=1}^3$
 if it is given by the total plane field $(T\cF^i)_{i=1}^3$.

Let $(\cF_0^i)_{i=1}^3$ be a total foliation on $S^3$
 with odd-framed $\cR$-components $R_1, \cdots, R_k$.
Let $M$ and $(\cF^i)_{i=1}^3$ denote
 the three-dimensional manifold and the total foliation
 obtained by the standard surgeries on all $R_i$'s,
 and $X$ the four-dimensional $2$-handlebody
 corresponding to the surgery as above.
By Proposition \ref{prop:2-handlebody},
 $X$ admits a unique spin structure $s_X$.
\begin{prop}
\label{prop:surgery formula}
The restriction of $s_X$ to $M=\del X$
 coincides with the one given by $(\cF^i)_{i=1}^3$.
\end{prop}
\begin{proof}
Let $h_j \subset X$ be the $2$-handles corresponding to $C(R_j)$
 for $j=1,\cdots,k$.
Total foliations $(\cF_0^i)_{i=1}^3$ and $(\cF^i)_{i=1}^3$
 define a spin structure $s_*$
 on a neighborhood of $S^3 \cup M =\del D^4 \cup \del X$
 in $X=D^4 \cup \bigcup_{j=1}^k h_j$,
 where $D^4$ is the four-dimensional ball.
Since $H_1(S^3,\ZZ_2)=0$, the sphere $S^3$ admits a unique spin structure.
It is known that it extends to $D^4$.
The closure of a connected component of $X\setminus (S^3 \cup M)$ is
 either the ball $D^4$ or a $2$-handle $h_j$.
Since they are homeomorphic to the four dimensional ball,
 the spin structure on $S^3 \cup M$ can be
 extended to $X$.
By the uniqueness of a spin structure on a $2$-handlebody,
 it completes the proof.
\end{proof}

\subsection{Gluing formula of the difference of Hopf invariant}
\label{sec:gluing}
For two total foliations $(\cF^i)_{i=1}^3$ and $(\cG^i)_{i=1}^3$
 which give the same spin structure,
 we denote the difference of Hopf invariant of the corresponding
 orthonormal frames (see Definition \ref{def:Hopf})
 by $H((\cF^i)_{i=1}^3,(\cG^i)_{i=1}^3)$.
\begin{dfn}
\emph{The positive total Reeb foliation} $(\cR_+^i)_{i=1}^3$
 is a total foliation on $S^3$
 which is the union of two $(-1)$-framed unknotted $\cR$-components.
\end{dfn}
Remark that each $\cR_+^i$ is a thick Reeb foliation
 and the cores of two $\cR$-components form a positive\footnote{Such a Reeb foliation is called
 \emph{a positive Reeb foliation}.
 The orientations given as the core of the $\cR$-component
 and given by the transverse orientation of $\cR_0^3$
 are opposite on one of the cores.} Hopf link
 under the transverse orientation of
 $\cR_+^3$.
Let $\tau_{S^3}$ be an orientation reversing diffeomorphism on $S^3$.
It is known that $H((\cR_+^i)_{i=1}^3,\tau_{S^3}(\cR_+^i)_{i=1}^3)=1$
 (see {\it e.g.} \cite[lemma 24]{Du2}).
By formulae  (\ref{eqn:Hopf 0}) and (\ref{eqn:Hopf})
 in page \pageref{eqn:Hopf}, we have
\begin{equation}
\label{eqn:inversion formula}
H((\tau_{S^3}(\cF^i))_{i=1}^3,(\cR_+^i)_{i=1}^3) 
 = -1-H((\cF^i)_{i=1}^3,(\cR_+^i)_{i=1}^3),
\end{equation}
 for any total foliation $(\cF^i)_{i=1}^3$ on $S^3$.

Let $(\cF^i)_{i=1}^3$ and $(\cG^i)_{i=1}^3$ be
 total foliations on $M$ and $S^3$, respectively.
Suppose that $(\cF^i)_{i=1}^3$ admits a null-homotopic
 $\cR$-component $R$
 and $(\cG^i)_{i=1}^3$ admits a $(-1)$-framed unknotted
 $\cR$-component $R'$.
Since both $R$ and $\cl{S^3 \setminus R'}$ are diffeomorphic to
 $S^1 \times D^2$,
 there exists a diffeomorphism $\psi:\cl{S^3 \setminus R'} \ra R$ such that
 $\psi_*(a_{\cR}(\del R'))=a_{\cR}(\del R)$ and
 $\psi_*(\mu_{R'})=\mu_R$.
Remark that the isotopy class of $\psi$ is uniquely determined.
By Proposition \ref{prop:gluing},
 there exists a total foliation
 $(\cF^i \cup_{R,R'} \cG^i)_{i=1}^3$ on $M$ such that
 it coincides with $(\cF^i)_{i=1}^3$ on $\cl{M \setminus R}$
 and with $(\psi(\cG^i))_{i=1}^3$ on $R$ up to isotopy.
\begin{prop}
\label{prop:gluing formula}
In the above situation, we have
\begin{equation}
\label{eqn:gluing formula}
H((\cF^i \cup_{R,R'} \cG^i)_{i=1}^3, (\cF^i)_{i=1}^3)
 = H((\cG^i)_{i=1}^3,(\cR_+^i)_{i=1}^3).
\end{equation}
\end{prop}
\begin{proof}
First, we notice that
 if two frames $\ch{e}$ and $\ch{e}'$
 on a three-dimensional manifold $M'$ gives the same spin structure,
 then $H(\ch{e},\ch{e}')$ is equal to
  the algebraic intersection number of the subsets
 $\{\ch{e}(p) \st p \in M'\}$ and  $\{-\ch{e}'(p) \st p \in M'\}$
 of the orthonormal frame bundle of $M'$.

For convenience, fix Riemannian metrics on $M$ and $S^3$ so that
 $\psi$ is an isometry between $\cl{S^3 \setminus R'}$ and $R$.
Let $\ch{e}_\cF$, $\ch{e}_\cG$, $\ch{e}_\cR$,
 $\ch{e}_*$
 be the orthonormal frames induced from
 $(\cF^i)_{i=1}^3$,  $(\cG^i)_{i=1}^3$, $(\cR_+^i)_{i=1}^3$,
 and $(\cF^i \cup_{R,R'} \cG^i)$, respectively.
By modifying $(\cR_+^i)_{i=1}^3$ in its isotopy class,
 we may assume that
 $R'$ is a $(-1)$-framed $\cR$-component of $(\cR_+^i)_{i=1}^3$ and
 $\psi(\ch{e}_\cR|_{\cl{S^3 \setminus R'}})=\ch{e}_\cF|_R$.
Take subsets $\Lambda=\{\ch{e}_\cG(p) \st p \in S^3\}$
 and $\Lambda'=\{-\ch{e}_\cR(p) \st p \in S^3\}$
 of the orthonormal frame bundle of $S^3$.
Let $\Fr\psi$ be the map between the frame bundles
 on $\cl{S^3 \setminus R'}$ and $R$ induced by $\psi$.
Then, we have
 $\Fr\psi(\Lambda)=\{\ch{e}_*(p) \st p \in R\}$
 and $\Fr\psi(\Lambda')=\{-\ch{e}_\cF(p) \st p\in R\}$.
Since $(\cF^i \cup_{R,R'} \cG^i)|_{\cl{M \setminus R}}
 =\cF^i|_{\cl{M \setminus R}}$,
 we also have
\begin{displaymath}
 \Fr\psi(\Lambda \cap \Lambda')
 = \{\ch{e}_*(p) \st p \in M\} \cap \{-\ch{e}_\cF(p) \st p \in M\}.
\end{displaymath}
It implies formula~(\ref{eqn:gluing formula}).
\end{proof} 

\section{Construction of Total foliations}\label{chap:construction}

\subsection{Braids in $W$}
\label{sec:braids}
Let $\SL(2,\ZZ)$ denote the group of $2 \times 2$-integer matrices
 with determinant one,
 and ${\Id}$ denote the identity matrix in $\SL(2,\ZZ)$.
Each element $A$ of $\SL(2,\ZZ)$ acts on $\TT^2$
 as a diffemorphism.

Fix $n \geq 1$ and put $Q_j=(j/n,j/n) +\ZZ^2 \in \TT^2$
 for $j=0,\cdots,n-1$.
\begin{dfn}
For $A \in \SL(2,\ZZ)$ and $n \geq 1$,
 we say $\Gamma\subset[0,1]\times\TT^2$ is
 \emph{a smooth $n$-braid twisted by $A$}
 if there exists a map $\gamma:\{0,\cdots,n-1\} \times [0,1] \ra \TT^2$
 and a permutation $\sigma$ on $\{0,\cdots,n-1\}$
 such that
\begin{itemize}
 \item  $\Gamma = \{(t,\gamma(j,t))
 \st (j,t) \in \{0,\cdots,n-1\} \times [0,1]\}$.
 \item $\gamma(j,t) \neq \gamma(j',t)$ for any $t \in [0,1]$
 if $j \neq j'$, and
 \item $\gamma(j,\varepsilon)=Q_j$
 and $\gamma(j,1-\varepsilon) =A \cdot Q_{\sigma(j)}$
 for any $j=0,\cdots,n-1$ and any sufficiently small $\varepsilon \geq 0$.
\end{itemize}
We call a subset $\Gamma^j=\{(t,\gamma(j,t)) \st t \in [0,1]\}$
 \emph{the $j$-th string of $\Gamma$}.

Let $B_n(A)$ be the set of all smooth $n$-braid twisted by $A$.
\end{dfn}
We can identify $B_n(A)$ with a subset of the set of smooth maps
 from $\{0,\cdots,n-1\} \times [0,1]$ to $\TT^2$.
This identification induces a topology on $B_n(A)$.
Let $\pi_0(B_n(A))$ be the set of connected components
 of $B_n(A)$.

For $A \in \SL(2,\ZZ)$, let $F_A$ be the diffeomorphism on $W$
 given by $F_A(t,w)=(t,A\cdot w)$.
We define $\tau_1(t,w)=(1-t,w)$,
 $\tau_-(t,w)=(t/2,w)$, and $\tau_+(t,w)=((1+t)/2,w)$
  for $(t,w) \in W=[0,1] \times \TT^2$.
\begin{dfn}
Let $\Gamma$ be a braid in $B_n(A)$.
\begin{itemize}
\item \emph{The inverse} $\Gamma^{-1} \in B_n(A^{-1})$
 is defined by $\Gamma^{-1}=F_{A^{-1}} \circ \tau_1(\Gamma)$.
\item \emph{The composition} $\Gamma * \Gamma' \in B_n(A'\cdot A)$
 for $\Gamma \in B_n(A)$ and $\Gamma' \in B_n(A')$
 is defined by $\Gamma* \Gamma'=\tau_-(\Gamma)
  \cup (F_A \circ \tau_+)(\Gamma')$.
\end{itemize}
\end{dfn}
They induce correspondent operations on $\pi_0(B_n(A))$.
We can see that they define a group structure on $\pi_0(B_n({\Id}))$,
 which is isomorphic to
 \emph{the braid group of $n$-strings on $\TT^2$}.
The composition also defines a free and transitive
 action of $\pi_0(B_n({\Id}))$ on $\pi_0(B_n(A))$.
In particular, each element of $\pi_0(B_n(A))$ gives a bijective map
 between $\pi_0(B_n({\Id}))$ and $\pi_0(B_n(A))$.

\subsection{Total foliations with braided leaves}\label{sec:W}
In this subsection, we fix an integer $n \geq 1$
 and a real number $\eta>0$ which is sufficiently smaller than $1/n$,
 for example, $\eta=(100n)^{-1}$.
Put $q_j=(j/n)+\ZZ \in S^1$ for $j=0,\cdots,n-1$.
Recall that $Q_j=(q_j,q_j) \in \TT^2$.

First, we define \emph{the standard total foliation}
 $(\cF_\std^i)_{i=1}^3$ on $W=[0,1] \times \TT^2$.
Let $(t,x,y)$ be the standard coordinate system of
 $W=[0,1] \times \TT^2$.
Fix a smooth function $\bar{\chi}_1$ on $\RR$ such that
 $0<\bar{\chi}_1(x)<\eta$ for $x \in ]1/16n,1/8n[$
 and $\bar{\chi}_1(x)=0$ otherwise.
Let $\chi_1$ be the function on $S^1$ given by
$\chi_1(q_j+x) = \bar{\chi}_1((1/2n)+x)-\bar{\chi}_1((1/2n)-x)$
 for any $j=0,\cdots,n-1$ and $x \in [0,1/n]$.
See Figure \ref{fig:chi1}.
\begin{figure}[ht]
\begin{center}
\includegraphics[scale=0.8]{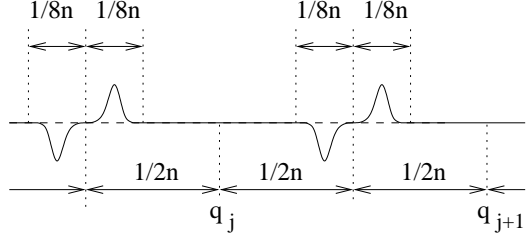}
\end{center}
\caption{Function $\chi_1$}
\label{fig:chi1}
\end{figure}
We define $1$-forms $\omega_\std^1$, $\omega_\std^2$,
 and $\omega_\std^3$ on $W$ by
\begin{align*}
\omega_\std^1(t,x,y) & = dy - \chi_1(y) dx, \\
\omega_\std^2(t,x,y) & = dx - \chi_1(x) dy, \\
\omega_\std^3(t,x,y) & = dt -
 (\bar{\chi}_1(t-3/8)+\bar{\chi}_1(t-5/8)) dy.
\end{align*}
Let $\cF^i_\std$ be the foliation generated
 by the kernel of $\omega^i_\std$ for $i=1,2,3$.
See Figure \ref{fig:standard}.
\begin{figure}[ht]
\begin{center}
\includegraphics[scale=1.0]{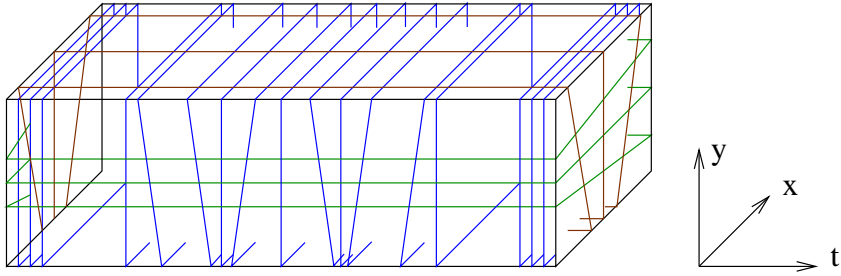}
\end{center}
\caption{Total foliation $(\cF_\std^i)_{i=1^3}$
 on $[0,1] \times [j/n,(j+1)/n]^2$}
\label{fig:standard}
\end{figure}
It is easy to check that
 the triple $(\cF^i_\std)_{i=1}^3$ is a total foliation.

\begin{dfn}
\label{dfn:turbularization}
Let $R$ be an embedded solid torus in $M$
 and $\cF$ a foliation on $\cl{M \setminus R}$.
We say a foliation $\cF_*$ is
 {\it obtained by a turbularization of $\cF$ along $R$}
 if $R$ is a thick Reeb component of $\cF_*$
 and there exists a diffeomorphism
 $\psi$ of the open manifold $M \setminus R$
 which is isotopic to the identity
 and satisfies $\cF_*|_{M\setminus R}=\psi(\cF)$.
\end{dfn}
Remark that if the restriction of $\cF$ to $R$
 is isotopic to the product foliation $\{\{pt\} \times D^2\}$,
 then we can turbularize $\cF$ along $R$.

Let $U_j$ be the interior of
 $[1/4,1/3] \times [q_j+(1/4n),q_j+(3/4n)] \times S^1$
 for $j=0,\cdots,n-1$ and
\begin{displaymath}
W_0=W \setminus \bigcup_{j=0}^{n-1} U_j.
\end{displaymath}
\begin{dfn}
We say a foliation $\cF_0$ on a subset $W'$ of $W$ is
 {\it almost horizontal} if
\begin{displaymath}
 T\cF_0(p) \subset \{v \in T_p W' \st dy(v)^2
 \leq \eta^{-2}(dt(v)^2+dx(v)^2)\}
\end{displaymath}
 for any $p \in W'$.
\end{dfn}

The next proposition shows how to make almost horizontal
 foliations part of a total foliation.
\begin{prop}
\label{prop:extension}
For any given almost horizontal foliation $\cF$ on $W_0$,
 there exists an extension $\cF^1$ of $\cF$ to $W$
 such that $(\cF^1,\cF_\std^2,\cF_\std^3)$ is a total foliation.
\end{prop}
\begin{proof}
Put
\begin{eqnarray*}
R_j^+ &=& \left\{\left.\left((3/8)+t, q_j+(1/2n)+x \right)
  \;\right|\; (t,x) \in D^2(1/8n) \right\} \times S^1,\\
R_j^- &=& \left\{\left.\left((5/8)+t, q_j+(1/2n)+x\right)
 \;\right|\; (t,x) \in D^2(1/8n) \right\} \times S^1
\end{eqnarray*}
 for $j=0,\cdots,n-1$.
\begin{figure}[ht]
\begin{center}
\includegraphics[scale=0.8]{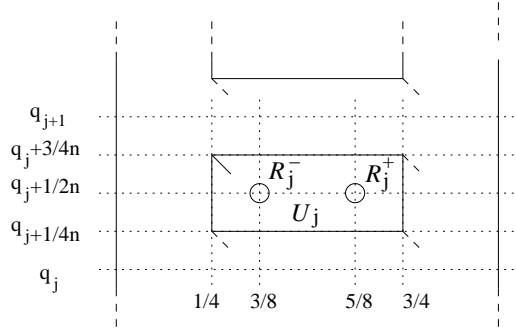}
\end{center}
\caption{The sets $W_0$ and $R_j^\pm$}
\label{fig:W0}
\end{figure}
Let $f_j$ be a diffeomorphism of $S^1$
 which is conjugate to the holonomy map of $\cF$
 along the torus $\del U_j$.
By $r_\alpha$,
 we denote the rigid rotation of angle $\alpha \in \RR$,
 {\it i.e.}, $r_\alpha(y)=y+\alpha$.
By a consequence of the Fundamental Theorem of Hermann
 (see {\it e.g.} \cite[Corollary 8.5.3]{CC}),
 there exist $\alpha_j^-, \alpha_j^+ \in \RR$
 and a diffeomorphism $g_j$ on $S^1$ such that
 $f_j=(g_j \circ r_{\alpha_j^-} \circ g_j^{-1}) \circ r_{\alpha_j^+}$
 for any $j=0,\cdots,n-1$.
It implies that we can extend $\cF$ to an
 almost horizontal foliation
 $\cG$ on $\cl{W\setminus\bigcup_{j=0}^{n-1}(R_j^- \cup R_j^+)}$
 such that the holonomy map of $\cG$
 along the torus $\del R_j^\sigma$
 is conjugate to the rigid rotation $r_{\alpha_j^\sigma}$
 for any $j=0,\cdots,n-1$ and $\sigma=\pm$.
Since $\cG$ is almost horizontal,
 it is transverse to $\cF_\std^2$ and $\cF_\std^3$.
A turbularization of $\cG$ along all $R_j^\pm$
 gives a foliation $\cF^1$ on $W$
 which is transverse to both $\cF^2_{\std}$ and $\cF^3_{\std}$.
See Figure \ref{fig:extension}.
\begin{figure}[ht]
\begin{center}
\includegraphics[scale=1.0]{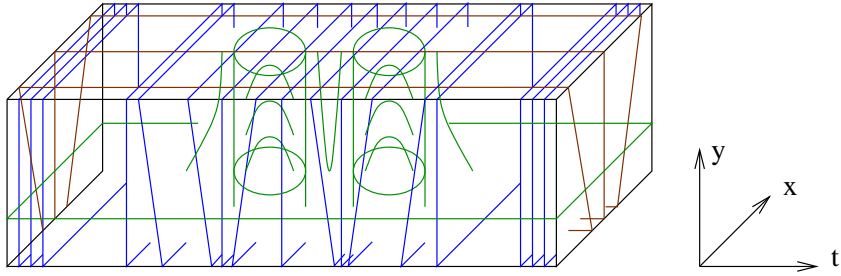}
\end{center}
\caption{Total foliation associated to an extension of $\cF_0$}
\label{fig:extension}
\end{figure}
\end{proof}

Recall that $F_A(t,w)=(t,A w)$, $\tau_1(t,w)=(1-t,w)$,
 $\tau_-(t,w)=(t/2,w)$, and $\tau_+(t,w)=((t+1)/2,w)$
 for $A \in \SL(2,\ZZ)$ and $(t,w) \in W$.
Let $(e_t,e_x,e_s)$ be the orthonormal frame on $W$
 which corresponds to the standard coordinates $(t,x,y)$.
\begin{dfn}
For $A \in \SL(2,\ZZ)$,
 let $\TFol(A)$ be the set of total foliations $(\cF^i)_{i=1}^3$
 on $W$ such that
\begin{itemize}
\item $\cF^3$ is transverse to $e_t$.
\item $\cF^i=\cF^i_\std$ on a neighborhood of $\{0\} \times \TT^2$
 for $i=1,2,3$,
\item $\cF^i=F_A(\cF^i_\std)$
 on a neighborhood of $\{1\} \times \TT^2$ for $i=1,2,3$
\end{itemize}
\end{dfn}
We introduce some operations on total foliations in $\TFol(A)$.
\begin{dfn}
Let $(\cF^i)_{i=1}^3$ and $(\cG^i)_{i=1}^3$ be total foliations
 in $\TFol(A)$ and $\TFol(A')$ respectively. 
\begin{itemize}
\item The {\it inverse} $((\cF^i)^{-1})_{i=1}^3 \in \TFol(A^{-1})$
 is defined by $(\cF^i)^{-1}=F_{A^{-1}} \circ \tau_1(\cF^i)$
 for $i=1,2,3$.
\item The {\it composition}
 $(\cF^i*\cG^i)_{i=1}^3 \in \TFol(A'A)$ of
 $(\cF^i)_{i=1}^3$ and $(\cG^i)_{i=1}^3$ is defined by
 $(\cF^i*\cG^i)|_{[0,1/2]}=\tau_-(\cF^i)$
 and $(\cF^i*\cG^i)|_{[1/2,1]}=(\tau_+ \circ F_{A}) (\cG^i)$.
\end{itemize}
\end{dfn}

We define an important subset of $\TFol(A)$ consisting of
 total foliations with {\it braided leaves}.
\begin{dfn}
\label{dfn:TFol(A)}
For $A \in \SL(2,\ZZ)$,
 we denote by $\TFol(A,n)$ the subset of $\TFol(A)$ 
 consisting of total foliations $(\cF^i)_{i=1}^3$
 such that
 $\Gamma =\bigcup_{j=0}^{n-1}(\cF^1 \cap \cF^2)(0,Q_j)$
 is an element of $B_n(A)$.
For $(\cF^i)_{i=1}^3 \in \TFol(A,n)$,
 we denote the connected component of $B_n(A)$
 containing the above $\Gamma$ by $\sigma((\cF^i)_{i=1}^3)$.
\end{dfn}
For any given $(\cF^i)_{i=1}^3 \in \TFol(A,n)$
 and $(\cG^i)_{i=1}^3 \in \TFol(A',n)$,
 it is easy to verify that
 $((\cF^i)^{-1})_{i=1}^3$ is an element of $\TFol(A^{-1},n)$
  with $\sigma(((\cF^i)^{-1})_{i=1}^3) = \sigma((\cF^i)_{i=1}^3)^{-1}$
 and
 $(\cF^i * \cG^i)_{i=1}^3$ is an element of $\TFol(A' A,n)$
  with $\sigma((\cF^i * \cG^i)_{i=1}^3)
 = \sigma((\cF^i)_{i=1}^3)*\sigma((\cG^i)_{i=1}^3)$.

Let $(\cF^i)_{i=1}^3$ be a total foliation in $\TFol(A,n)$.
Put $\Gamma^j=\cF^1 \cap \cF^2(0,Q_j)$
 for $j=0,\cdots,n-1$.
For each $k=1,2$ and each $j=0,\cdots, n-1$,
  there exists a smooth function $\theta_k^j$ on $\Gamma^j$ such that
\begin{displaymath}
 \cos(2\pi \theta_k^j(p))e_x(p) + \sin(2\pi \theta_k^j(p))e_y(p)
 \in T \cF^k(p)
\end{displaymath}
 for any $p \in \Gamma^j$.
We define {\it the rotation $\Theta_k((\cF^i)_{i=1}^3,j)$
 of $\cF^k$ along the $j$-th string} by
\begin{displaymath}
 \Theta_k((\cF^i)_{i=1}^3,j)=\theta_k^j(1,w^j_1)-\theta_k^j(0,w^j_0),
\end{displaymath}
 where $\{(0,w^j_0),(1,w^j_1)\}=\del \Gamma^j$.
It does not depend on the choice of $\theta_k^j$.

For any sufficiently small $\delta>0$ and $j=0,\cdots,n-1$
 there exist two maps $f$ and $g$ from $[-2\delta,2\delta]$ to $\RR$
 such that the holonomy of $\cF^1 \cap \cF^2$ along $\Gamma^j$
 is given by the map
 $(0,w^j_0+(x,y)) \mapsto (1,w^j_1+ A \cdot (f(x),g(y)))$.
We define {\it the $\delta$-normalized holonomy
 of $\cF^1 \cap \cF^2$ along $j$-th string} by
 the pair $(H_x^\delta((\cF^i)_{i=1}^3,j),
 H_y^\delta((\cF^i)_{i=1}^3,j))$ of maps from $[-2,2]$ to $\RR$
 given by
\begin{displaymath}
 H_x^\delta((\cF^i)_{i=1}^3,j)(x)=\delta^{-1} \cdot f(\delta x),\hsp
 H_y^\delta((\cF^i)_{i=1}^3,j)(y)=\delta^{-1} \cdot g(\delta y).
\end{displaymath}
Let $\Diff_0([-2,2],0)$ denote the set of diffeomorphism $f$ on $[-2,2]$
 such that $f(0)=0$ and $\cl{\{f(x) \neq x\}}\subset ]-2,2[$.
\begin{prop}
\label{prop:fol}
For any $A \in \SL(2,\ZZ)$, $\sigma \in \pi_0(B_n(A))$, $m \in \ZZ$,
 and any sequences $(f_j)_{j=0}^{n-1}$ and $(g_j)_{j=0}^{n-1}$
 in $\Diff_0([-2,2],0)$,
 there exists $(\cF^i)_{i=1}^3 \in \TFol(A,n)$
 and $\delta>0$ such that
\begin{itemize}
 \item $\sigma((\cF^i)_{i=1}^3)=\sigma$,
 \item $\Theta((\cF^i)_{i=1}^3,j)$ does not depend on $j$
 and belongs to the interval $[m,m+1[$, and
 \item $H_x^\delta((\cF^i)_{i=1}^3,j)=f_j$
 and $H_y^\delta((\cF^i)_{i=1}^3,j)=g_j$ for any $j=0,\cdots,n-1$.
\end{itemize}
\end{prop}
The rest of the subsection is devoted to the proof of the proposition.
We divide it into several lemmas.
Put 
\begin{equation}
\label{eqn:matrix definition}
 A_{xy}=
 \left(\begin{array}{cc}
 0 & 1 \\ 1 & 0 \end{array}
 \right),
\;
 A_1=\left(
\begin{array}{cc}
 1 & 0 \\ 1 & 1
\end{array}\right),
\;
 A_2=\left(
\begin{array}{cc}
 1 & 1 \\ 0 & 1
\end{array}\right),
\;
 A_*=\left(
\begin{array}{cc}
 0 & -1 \\ 1 & 1
\end{array}\right).
\end{equation}
They satisfy the following equations:
\begin{equation}
\label{eqn:matrix identity}
A_{xy}^2 = {\Id},
\;
A_{xy} \cdot A_1 \cdot A_{xy}=A_2,
\;
A_*=A_2^{-1} \cdot A_1,
\;
A_*^3=-{\Id}.
\end{equation}

\begin{lemma}
\label{lemma:transpose}
The triple
 $(F_{A_{xy}}(\cF^2),F_{A_{xy}}(\cF^1),F_{A_{xy}}(\cF^3))$
 is a total foliation in $\TFol(A_{xy} \cdot A \cdot A_{xy})$
 for any $(\cF^i)_{i=1}^3 \in \TFol(A)$.
Moreover, if $(\cF^i)_{i=1}^3 \in \TFol(A,n)$,
 then the above triple is in
 $\TFol(A_{xy} \cdot A \cdot A_{xy},n)$.
\end{lemma}
\begin{proof}
It is an easy consequence of
 the identities $F_{A_{xy}}(\cF^1_\std)=\cF^2_\std$ and
 $F_{A_{xy}}(\cF^2_\std)=\cF^1_\std$.
\end{proof}

Let $\sigma_0$ be the connected component of $B_n({\Id})$
 represented by the constant braid
 $\Gamma_0=[0,1] \times \{Q_1,\cdots,Q_{n-1}\}$.
The following lemma is an interpretation of the construction
 in \cite[p.49--50]{Ha} in our setting.
\begin{lemma}
\label{lemma:holonomy}
For any given $\delta_0>0$
 and sequences $(f_j)_{j=0}^{n-1}$ and $(g_j)_{j=0}^{n-1}$
 in $\Diff_0([-2,2],0)$
 there exist $(\cF^i)_{i=1}^3 \in \TFol({\Id},n)$
 and $\delta \in (0,\delta_0)$ such that
 $\sigma((\cF^i)_{i=1}^3)=\sigma_0$,
 $\Theta_1((\cF^i)_{i=1}^3,j)=\Theta_2((\cF^i)_{i=1}^3,j)=0$,
 $H_x^\delta((\cF^i)_{i=1}^3,j)=f_j$, and
 $H_y^\delta((\cF^i)_{i=1}^3,j)=g_j$ for any $j=0,\cdots,n-1$.
\end{lemma}
\begin{proof}
Take $\delta \in (0,\min\{\delta_0,1\})$.
First, we fix $j_* \in \{0,\cdots,n-1\}$ and
 a diffeomorphism $g \in \Diff_0([-2,2],0)$
 and we show the lemma
 for the case all $f_j$'s and $g_j$'s are the identity
 except $g_{j_*}=g$.
Let us modify $\cF^1_\std$
 so as to have the holonomy corresponding to $g$.
Take a smooth map $\chi_2: S^1 \times [1/4,3/4] \ra S^1$
 such that
\begin{enumerate}
\item $\chi_2(y,1/4+\epsilon)=y$
 and $\chi_2(y,3/4-\epsilon)=\chi_2(y,3/4)$
 for any $y \in [0,1]$ and any small $\epsilon \geq 0$,
\item $\chi_2(y,t)=y$ if $y \not\in [q_{j_*}-2\delta,q_{j_*}+2\delta]$,
\item $\chi_2(q_{j_*}+y',3/4)=q_{j_*}+\delta g(\delta^{-1}y')$
 for any $y' \in [-\delta,\delta]$,
\item $\frac{\del h}{\del y}(y,t)>0$ and
 $\left|\frac{\del h}{\del t}(y,t)\right| < \eta^{-1}$ for any $(y,t)$.
\end{enumerate}
Remark that $\chi_2(\cdot,t)$ is a diffeomorphism of $S^1$
 for any $t \in [1/4,3/4]$.

Put $J_{j_*}=[q_{j_*}-(1/4n), q_{j_*}+(1/4n)] \subset S^1$ and
 $V_{j_*}=[1/4,3/4] \times J_{j_*} \times S^1$.
Since $\del V_{j_*} \cap \Int W_0
 \subset \{1/4,3/4\} \times J_{j_*} \times S^1$,
 we can define a foliation $\cF^1_0$ on $W_0$ such that
 $\cF^1_0|_{W_0 \setminus V_{j_*}}=\cF^1_\std$ and
\begin{equation*}
 (\cF^1_0|_{V_{j_*}})(1/4,x,y)
 =\left\{(t,x',\chi_2(y,t)) \st (t,x') \in
  [1/4,3/4] \times J_{j_*}\right\}
\end{equation*}
 for any $(x,y) \in J_{j_*} \times [0,1]$.
See Figure \ref{fig:holonomy}.
\begin{figure}[ht]
\begin{center}
\includegraphics[scale=1.0]{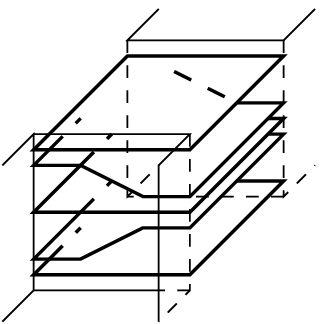}
\end{center}
\caption{Foliation $(\cF_0^1|_{V_{j_*}})$}
\label{fig:holonomy}
\end{figure}
Since $\left|\frac{\del h}{\del t}(y,t)\right| <\eta^{-1}$ for any $(y,t)$,
 the foliation $\cF_0^1$ is almost horizontal.
By Proposition \ref{prop:extension},
 there exists a total foliation $(\cF^i)_{i=1}^3 \in \TFol(\Id)$
 such that $\cF^1|_{W_0}=\cF^1_0$ and $\cF^i=\cF^i_\std$ for $i=2,3$.
Since $\Gamma_0$ is tangent to $\cF_0^1 \cap \cF_0^2$,
 $(\cF^i)_{i=1}^3$ is contained in $\TFol(\Id,n)$.
The holonomy of $\cF^1 \cap \cF^2$ along
 the $j_*$-th string $\Gamma_0^{j_*}$ is 
\begin{equation*}
 (0,q_{j_*}+x,q_{j_*}+y) \mapsto (1,q_{j_*}+x,\chi_2(q_{j_*}+y,3/4)) 
\end{equation*}
 for $(x,y) \in [-\delta,\delta]^2$.
Hence, $H_x^\delta((\cF^i)_{i=1}^3,j_*)$ is the identity map
 and $H_y^\delta((\cF^i)_{i=1}^3,j_*)=g$.
It is easy to see that
 $H_x^\delta((\cF^i)_{i=1}^3,j)$ and $H_y^\delta((\cF^i)_{i=1}^3,j)$
 are the identity maps for all $j \neq j_*$,
 and $\Theta_1((\cF^i)_{i=1}^3,j)=\Theta_2((\cF^i)_{i=1}^3,j)=0$
 for any $j=0,\cdots,n-1$.

By Lemma \ref{lemma:transpose},
 the total foliation
 $(F_{A_{xy}}(\cF^2),F_{A_{xy}}(\cF^1),F_{A_{xy}}(\cF^3))$
 is contained in $\TFol({\Id},n)$.
It easy to verify that it satisfies the required conditions
 for the case $f_{j_*}=g$ and all the other $f_j$'s and $g_j$'s are
 the identity map.
Hence, we can obtain the required total foliation
 for a general sequence $(f_j,g_j)_{j=0}^{n-1}$
 as a composition of the total foliations
 given by the above construction.
\end{proof}

\begin{lemma}
\label{lemma:braid}
For any given $\sigma \in \pi_0(B_n({\Id}))$,
 there exists $(\cF^i)_{i=1}^3 \in \TFol({\Id},n)$
 such that
 $\sigma((\cF^i)_{i=1}^3)=\sigma$ and
 $\Theta_1((\cF^i)_{i=1}^3),j)=\Theta_2((\cF^i)_{i=1}^3),j)=0$
 for any $j=0,\cdots,n-1$.
\end{lemma}
\begin{proof}
Fix a smooth function $\alpha$ on $[0,1]$ such that
 $\alpha(t)=0$ for $t \in [0,1/4]$,
 $\alpha(t)=1/n$ for $t \in [3/4,1]$,
 and $0 \leq d\alpha/dt(t) \leq \eta^{-1}$ for any $t \in [0,1]$.
Put $V_j=[0,1] \times [q_j-(1/4n),q_j+(1/4n)] \times S^1$
 and $\Gamma^j(y)=\{(t,q_j,y+\alpha(t)) \st t \in [0,1]\}$
 for $j =0,\cdots,n-1$ and $y \in S^1$.

First, for any given $m=0,\cdots, n-1$,
 there exists $(\cF^i_m)_{i=1}^3 \in \TFol({\Id})$
 such that
\begin{itemize}
 \item $\cF^1_m|_{W_0 \setminus V_m}=\cF_\std^1|_{W_0 \setminus V_m}$,
  $\cF^2_m=\cF^2_\std$, and
 \item $\Gamma^m(y)$ is tangent to $\cF^1_m \cap \cF^2_m$
 for any $y \in S^1$.
\end{itemize}
In fact, it can be obtained 
 by the same construction as the total foliation $(\cF^i)_{i=1}^3$
 in the proof of Lemma \ref{lemma:holonomy}
 by replacing $\chi_2(y,t)$
 in the definition of $\cF_0^1$ with $y+\alpha(t)$.

Put $\cG_m^1=F_{A_{xy}}(\cF^2_m)$, $\cG^2_m=F_{A_{xy}}(\cF_m^1)$,
 and $\cG_m^3=F_{A_{xy}}(\cF^3_2)$.
Let $((\cF_m^i)^{-1})_{i=1}^3$ and $((\cG_m^i)^{-1})_{i=1}^3$
 be the inverses of $(\cF_m^i)_{i=1}^3$ and $(\cG_m^i)_{i=1}^3$
 respectively.
Remark that all of them are total foliation in $\TFol({\Id})$
 by Lemma \ref{lemma:transpose}.
We define $(\cF^i_{\sigma_m})_{i=1}^3 \in \TFol({\Id})$ by
\begin{displaymath}
 \cF^i_{\sigma_m}=
 \cF_m^i * (\cF_{m+1}^i)^{-1} * \cG_{m+1}^i * (\cG_m^i)^{-1}
\end{displaymath}
 and put $\sigma_m=\sigma((\cF_{\sigma_m}^i)_{i=1}^3)$
 for $m=0,\cdots,n-2$.
Then, $(\cF_{\sigma_m}^i)_{i=1}^3$ is a total foliation
 in $\TFol({\Id},n)$
 and $\sigma_m$ represents
 a half twist of $m$-th and $(m+1)$-st strings.
See Figure \ref{fig:braid}.
\begin{figure}[ht]
\begin{center}
\includegraphics[scale=1.0]{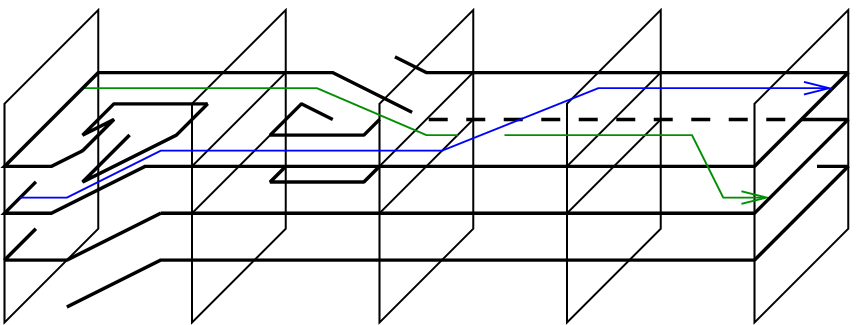}
\end{center}
\caption{Proof of Lemma \ref{lemma:braid}}
\label{fig:braid}
\end{figure}

Let $(\cF_{\rho_m}^i)_{i=1}^3$ and $(\cF_{\tau_m}^i)_{i=1}^3$
 be $n$-times compositions $(\cF_m^i*\cdots *\cF_m^i)_{i=1}^3$
 and $(\cG_m^i * \cdots * \cG_m^i)_{i=1}^3$ respectively.
Put $\rho_m=\sigma((\cF_{\rho_m}^i)_{i=1}^3)$
 and $\tau_m=\sigma((\cF_{\tau_m}^i)_{i=1}^3)$.
We can see that both $(\cF_{\rho_m}^i)_{i=1}^3$
 and $(\cF_{\tau_m}^i)_{i=1}^3$
 are total foliations in $\TFol({\Id},n)$ and $\rho_m$
 ({\it resp.} $\tau_m$) is represented by
 a braid such that the $m$-th string winds once
 in the $y$-({\it resp.} $x$-)direction
 and other strings are fixed.

It is easy to see that $\Theta_k((\cF_\sigma^i)_{i=1}^3,j)=0$
 for any $k=1,2$, $m=0,\cdots,n-1$,
 and $\sigma=\sigma_m,\rho_m,\tau_m$.
It is known that $\{\sigma_m, \rho_m,\tau_m \st m=0,\cdots, n-1\}$
 generates $\pi_0(B_n({\Id}))$ (see {\it e.g.} \cite{Bi} or \cite{Go}).
Hence, we can obtain the required total foliation
 as a composition of the total foliations constructed above
 and their inverses.
\end{proof}

\begin{lemma}
\label{lemma:matrix}
There exist $(\cF_1^i)_{i=1}^3 \in \TFol(A_1,n)$ satisfying
\begin{equation}
\label{eqn:matrix}
 \Theta_1((\cF_1^i)_{i=1}^3,j)=1/8,\hsp
 \Theta_2((\cF_1^i)_{i=1}^3,j)=0
\end{equation}
 for any $j=0,\cdots,n-1$.
\end{lemma}
\begin{proof}
Take a smooth map $\bar{\chi}_3:[0,1] \ra \RR$ such that
\begin{itemize} 
 \item $0 \leq \frac{d \bar{\chi}_3}{d x}(x) \leq \eta^{-1}$ 
 for any $x \in [0,1]$,  
 \item $\bar{\chi}_3(x)=0$ for any $x \in [0,1-(9/16n)]$
 and $\bar{\chi}_3(x)=1$ for any $x \in [1-(7/16n),1]$.
\end{itemize}
It induces a map $\chi_3:S^1 \ra S^1$ of degree $1$.
We define a diffeomorphism $G$ of $W_0$ by
 $G(t,x,y)=(t,x,y+\chi_3(x))$ if $t \in [3/4,1]$
 and $G(t,x,y)=(t,x,y)$ otherwise.
It is well-defined and satisfies
 $G(\cF^i_\std|_{W_0})=\cF^i_\std|_{W_0}$ for $i=2,3$.
Since $G(\cF^1_\std|_{W_0})$ is almost horizontal,
 Proposition \ref{prop:extension} implies that
 there exists an extension $\cG$ of $G(\cF^1_\std|_{W_0})$ to $W$
 which is transverse to $\cF_\std^2$ and  $\cF_\std^3$.
Remark that the constant braid $\Gamma_0$
 is tangent to $\cG \cap \cF_\std^2$.

Since $\chi_3(x)-x$ is a map of degree $0$,
 we can a smooth function $\alpha$ on $S^1 \times [0,1]$ such that
 $\alpha(x,t)=0$ for $(x,t) \in S^1 \times [0,3/4]$ and
 $x=\chi_3(x)+\alpha(x,t)$ for $(x,t) \in S^1 \times [7/8,1]$.
We define a diffeomorphism $\bar{G}$ of $W$ by
 $\bar{G}(t,x,y)=(t,x,y+\alpha(x,t))$.
Remark that $\bar{G} \circ G(t,x,y)=(t,x,y)$
 if $t \in [0,3/4]$ and
 $\bar{G} \circ G(t,x,y)=(t,A_1(x,y))$ if $t \in [7/8,1]$.
Put $\cF_1^1=\bar{G}(\cG)$,
 $\cF_1^i=\bar{G}(\cF_\std^i)$ for $i=2,3$
 and $\Gamma=\bar{G}(\Gamma_0)$.
Then, $(\cF_1^i)_{i=1}^3$ is a total foliation
 contained in $\TFol(A_1)$
 and $\bar{G}(\Gamma_0)$ is a braid in $B_n(A_1)$
 which is tangent to $\cF_1^1 \cap \cF_1^2$.
Therefore, $(\cF_1^i)_{i=1}^3$ an element of $\TFol(A_1,n)$.
See Figure \ref{fig:matrix}.
\begin{figure}[ht]
\begin{center}
\includegraphics[scale=1.0]{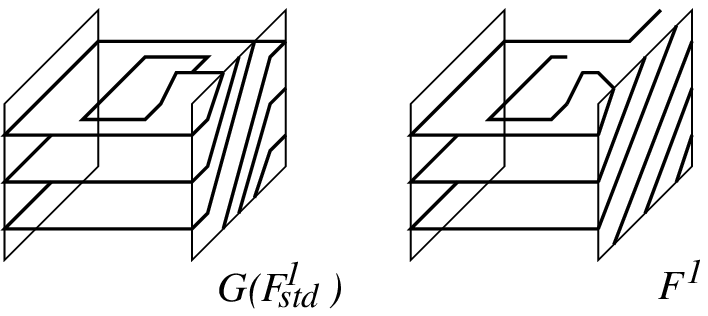}
\end{center}
\caption{Proof of Lemma \ref{lemma:matrix}}
\label{fig:matrix}
\end{figure}
Since $\cF_1^1$ is almost horizontal on $W_0$,
 we have $\Theta_1((\cF_1^i)_{i=1}^3,j)=1/8$
 for any $j=0,\cdots,n-1$.
By the transversality of $\cF_1^1$ and $\cF_1^2$,
 we also have $\Theta_2((\cF_1^i)_{i=1}^3,j)
 -\Theta_1((\cF_1^i)_{i=1}^3,j) \in ]-1/2,1/2[$.
It implies $\Theta_2((\cF_1^i)_{i=1}^3,j)=0$
 for any $j=0,\cdots,n-1$.
\end{proof}

Starting from the total foliation $(\cF_1^i)_{i=1}^3$
 in Lemma \ref{lemma:matrix},
 we define a total foliation $(\cF_2^i)_{i=1}^3$ 
 by $\cF_2^1=F_{A_{xy}}(\cF_1^2)$, $\cF_2^2=F_{A_{xy}}(\cF_1^1)$
 and $\cF_2^3=F_{A_{xy}}(\cF_1^3)$.
By Lemma \ref{lemma:transpose}
 and the third equation of (\ref{eqn:matrix identity}),
 we have $(\cF_2^i)_{i=1}^3 \in \TFol(A_2)$.
By (\ref{eqn:matrix}), we also have
\begin{equation}
\label{eqn:matrix 2}
 \Theta_1((\cF_2^i)_{i=1}^3,j)=0,\hsp
 \Theta_2((\cF_2^i)_{i=1}^3,j)=-1/8.
\end{equation}
 for any $j=0,\cdots,n-1$.

\begin{lemma}
\label{lemma:rotation}
For any $m \in \ZZ$,
 there exists $(\cF^i)_{i=1}^3 \in \TFol({\Id},n)$ such that
 $\Theta_1((\cF^i)_{i=1}^3,j)=\Theta_2((\cF^i)_{i=1}^3,j)=m$
 for any $j=0,\cdots,n-1$.
\end{lemma}
\begin{proof}
Let $((\cF_2^i)^{-1})_{i=1}^3$ be the inverse of $(\cF_2^i)_{i=1}^3$.
Put $\cG^i=\cF_1^i * (\cF_2^i)^{-1}$ for $i=1,2,3$.
Since $A_*=A_2^{-1} \cdot A_1$,
 we have $(\cG^i)_{i=1}^3 \in \TFol(A_*,n)$.
The equations (\ref{eqn:matrix}) and (\ref{eqn:matrix 2}),
 we also have $\Theta_1((\cG^i)_{i=1}^3,j)=1/4$
 and $\Theta_2((\cG^i)_{i=1}^3,j)=1/8$ for any $j=0,\cdots,n-1$.
Let $(\cG_k^i)_{i=1}^3$ be 
 the $k$-times composition of $(\cG^i)_{i=1}^3.$
Since $A_*^3=-{\Id}$,
 we have $\Theta_1((\cG_3^i)_{i=1}^3,j)=1/2$
 and $\Theta_2((\cG_3^i)_{i=1}^3,j)=1/2$ for any $j=0,\cdots,n-1$.
Hence, $(6m)$-times composition $(\cG^i_{6m})_{i=1}^3$
 of $(\cG^i)_{i=1}^3$ is the required total foliation for $m \geq 0$.
For $m <0$,
 it is sufficient to take the inverse
 $((\cG^i_{|6m|})^{-1})_{i=1}^3$ of $(\cG_{|6m|}^i)_{i=1}^3$.
\end{proof}

\begin{lemma}
\label{lemma:SL}
For any $A \in \SL(2,\ZZ)$,
 there exist $(\cF^i)_{i=1}^3 \in \TFol(A,n)$ such that
 $\Theta_1((\cF^i)_{i=1}^3,j)$ and $\Theta_2((\cF^i)_{i=1}^3,j)$
 does not depends on $j$.
\end{lemma}
\begin{proof}
It is an immediate consequence of Lemma \ref{lemma:matrix} 
 and the fact that $\{A_1,A_2\}$ generates $\SL(2,\ZZ)$.
\end{proof}

Finally, Proposition \ref{prop:fol} is an immediate consequence
 of Lemmas \ref{lemma:holonomy}, \ref{lemma:braid},
 \ref{lemma:rotation}, and \ref{lemma:SL}.

\subsection{Braided knots in embedded solid tori}
Let $\psi$ be an embedding from $Z=S^1 \times D^2$  to
 an oriented three-dimensional manifold $M$.
We say $\psi$ is a {\it $0$-framed null-homotopic embedding}
 if $K_0=\psi(S^1 \times \{(0,0)\})$ is null-homotopic
 and $\psi$ is a $0$-framed tubular coordinate of $K_0$.
We also say $\psi$ is {\it unknotted}
 if $K_0$ is unknotted.

We say a smooth link $L$ in $M$ is {\it $\psi$-braided}
 if $L \subset \psi(Z)$ and
 $\psi^{-1}(L)$ is transverse to the production
 foliation $\{t \times D^2\}_{t \in S^1}$.
Remark that any component of $L$ is a $\psi$-braided knot.
\begin{dfn}
Let $L$ be a $\psi$-braided oriented knot or link.
We denote by $n(L;\psi)$
 be the cardinality of $\psi^{-1}(L) \cap (0 \times D^2)$.
We define {\it the $(\psi,n)$-framing} of $L$ in $M$ by a vector field
\begin{equation*}
 v_n(\psi(t,w)) =D\psi(\cos(2\pi n t)e_x(t,w)+\sin(2\pi n t)e_y(t,w)) 
\end{equation*}
 for $(t,w) \in \psi^{-1}(L)$,
 where $(e_t,e_x,e_y)$ is the standard frame on
 $S^1 \times D^2 \subset S^1 \times \RR^2$.
\end{dfn}
Remark that $(\psi,n)$-framing of $K$ may not be the $n$-framing
 (in the sense of Definition \ref{dfn:framing})
 even if $\psi$ is $0$-framed and unknotted.
See Lemma \ref{lemma:frame of braid}.

Let $\psi_0$ be a $0$-framed unknotted embedding of $Z$ into $\RR^3$
 defined by $\psi_0(t,x,y)=((x+2)\cos 2\pi t,(x+2)\sin 2\pi t,y)$
 and $P_{xy}$ denote the projection from $\RR^3$ to $\RR^2$
 given by $P_{xy}(x,y,z)=(x,y)$.
For any given $0$-framed unknotted embedding $\psi$ of $Z$ into $M$,
 we can take an embedding $\varphi$ of $\RR^3$ into $M$
 so that $\varphi^{-1} \circ \psi=\psi_0$.
Take a $\psi$-braided link $L$ in $M$.
The map $\varphi$ can be perturbed into another embedding $\varphi_1$
 such that the map $P_{xy} \circ \varphi_1$ is a regular projection
 associated with $\varphi_1^{-1}(L)$.
See {\it e.g.} \cite{Ro} for the definition of
 a regular projection and a link diagram.
For any component $K$ of $L$,
 let $\omega_\pm(K;\psi)$ be the number of positive
 and negative crossings in the diagram $P_{xy} \circ \varphi_1^{-1}(K)$.
See Figure \ref{fig:braided}.
\begin{figure}[ht]
\begin{center}
\includegraphics[scale=1.0]{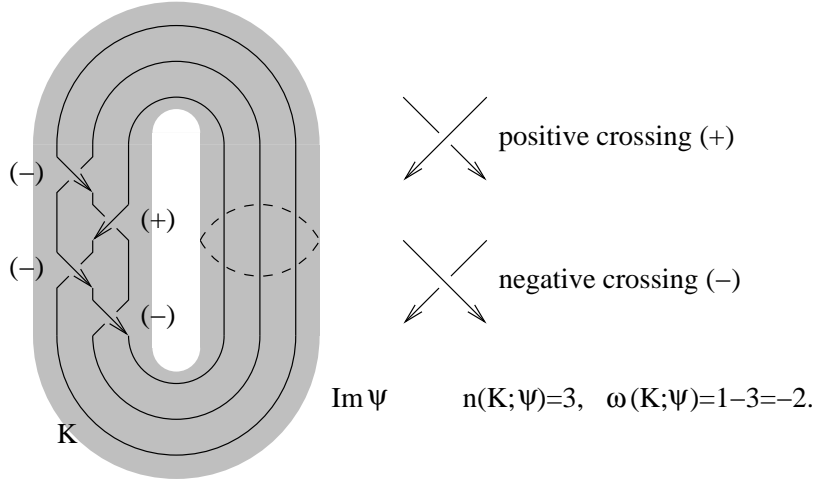}
\end{center}
\caption{A link diagram of a braided knot}
\label{fig:braided}
\end{figure}
We put $\omega(K;\psi)=\omega_+(K;\psi)-\omega_-(K;\psi)$.
Remark that $\omega(K;\psi)$ and $n(K;\psi)$ depend only on
 the isotopy class of $K$ as a $\psi$-braided knot.

We show two lemmas,
 which give relations
 between $n(K;\psi)$, $\omega(K;\psi)$ and the framing of $K$.
\begin{lemma}
\label{lemma:odd framing}
Let $\psi$ be a $0$-framed unknotted embedding from $Z$ to $M$
 and $K$ be a $\psi$-braided knot in $M$. 
Then, $\omega(K;\psi)+n(K;\psi)$ is odd. 
\end{lemma}
\begin{proof}
Since $K$ is connected, it induces a cyclic permutation
 on the $n(K;\psi)$-points set $\psi^{-1}(K) \cap (0 \times D^2)$.
Then, the signature of the permutation is $(-1)^{n(K;\psi)+1}$.
Since the induced permutation is the product
 of $(\omega_+(K;\psi)+\omega_-(K;\psi))$ transpositions,
 its signature is also $(-1)^{(\omega_+(K;\psi)+\omega_-(K;\psi))}$.
In particular, $n(K;\psi)-(\omega_+(K;\psi)+\omega_-(K;\psi))$
 is odd.
Hence, also $\omega(K;\psi)+n(K;\psi)$ is.
\end{proof}

\begin{lemma}
\label{lemma:frame of braid}
Let $\psi$ be a $0$-framed unknotted embedding from $Z$ to $M$
 and $K$ be a $\psi$-braided knot in $M$.
Then, the $(\psi,m)$-framing of $K$
 coincides with
 the $(\omega(K;\psi)+m \cdot n(K;\psi))$-framing of $K$
 as a null-homotopic knot in $M$.
\end{lemma}
\begin{proof}
Suppose that the $(\psi,0)$-framing of $K$ is the $n_0$-framing.
It is easy to see that the
 $(\psi,m)$-framing is $n_0+m \cdot n(K;\psi)$.
Under the identification of $\psi(Z)$
 and the standard torus $\psi_0(Z)$,
 the $(\psi,0)$-framing gives \emph{the blackboard framing},
 that is, the one transverse to the projection to the link diagram.
By a well-known result in knot theory
 (see {\it e.g.} \cite[Proposition 4.5.8]{GS}),
 it coincides with the $\omega(K;\psi)$-framing of $K$.
Hence, we have $n_0=\omega(K;\psi)$.
\end{proof}

\subsection{The trefoil complement}\label{sec:trefoil}
In this subsection, we construct a total foliation on $S^3$
 containing $\cR$-components
 such that their cores form an arbitrary given link.
It will be done by using the fibration of the complement of the trefoil.
Note that the same construction can be done
 for other fibered knot with one-punctured torus fibers,
 {\it e.g.} the figure-eight knot.

Let $A_*$ be the matrix defined in (\ref{eqn:matrix definition})
 and $M_*$ be the mapping torus $W/(0,w) \sim (1, A_* \cdot w)$
 of the linear map defined by $A_*$.
By $P_{M_*}$, we denote the natural projection from $W$ to $M_*$.
Since any total foliation $(\cF^i)_{i=1}^3 \in \TFol(A_*)$
 is compatible with the projection $P_{M_*}$ at $\del W$,
 we can define a total foliation $(P_{M_*}(\cF^i))_{i=1}^3$ on $M_*$
 such that $P_{M_*}(\cF^i)(P_{M_*}(t,w))=P_{M_*}(\cF^i(t,w))$
 for any $i=1,2,3$ and $(t,w) \in W$.

Since $Q_0=(0,0)+\ZZ^2$ is a fixed point of $A_*$,
 $P_{M_*}([0,1] \times Q_0)$ is a knot in $M_*$.
We denote it by $K_0$.
Fix an embedding $\psi_{K_0}:Z \ra M_*$
 such that $\psi_{K_0}(S^1 \times \{(0,0)\})=K_0$ and
 $\psi_{K_0}(t \times D^2) \subset P_{M_*}(t \times D^2(1/8n))$.

\begin{prop}
\label{prop:fol in W*}
For any $\psi_{K_0}$-braided link $L$ and $m \in \ZZ$,
 there exists a total foliation $(\cF^i)_{i=1}^3$ on $M_*$
 such that each component of $L$ is the core of an $\cR$-component
 and its framing determined by $\cF^1$ is the $(\psi_{K_0},m)$-framing.
\end{prop}
\begin{proof}
We take a smooth function $\theta$ on $[0,1]$ such that
$D\psi(e_x(t,0,0))$ is parallel to
 $DP_{M_*}(\cos \theta(t) e_x(t,0,0) +\sin \theta(t) e_y(t,0,0))$.
For $m \in \ZZ$, we define a vector field $\bar{v}_m$ on $W$ by
\begin{align*}
\bar{v}_m(t,w)
 & =\cos(2\pi mt+\theta(t))e_x(t,w) +\sin(2\pi mt+\theta(t))e_y(t,w)
\end{align*}
 for $(t,w) \in W$.
Then, the vector field $v_m=DP_{M_*}(\bar{v}_m)$ on $M_*$
 is well-defined and
 the restriction of $v_m$ to a knot $K$ in $M_*$
  gives the $(\psi_{K_0},m)$-framing of $K$.

Take an isotopy $\{F_s\}_{s \in [0,1]}$ of $M_*$
 so that $F_0$ is the identity map,
 $F_s(P_{M_*}(t \times \TT^2))=P_{M_*}(t \times \TT^2)$
 for any $s \in [0,1]$ and $t \in S^1$,
 and $F_1(L) \in B_n(A_*)$ for some $n \geq 1$.
Put $L'=F_1(L)$.
Since $\{(F_s(L),v_m)\}_{s \in [0,1]}$
 is an isotopy between framed knots $(L,v_m)$ and $(L',v_m)$,
 it is sufficient to find a total foliation $(\cF^i)_{i=1}^3$
 such that each component of $L'$ is the core of an $\cR$-component
 and its framing determined by $\cF^1$
 coincides with the one represented by $v_m$.

Take a subset $S_L$ of $\{0,\cdots,n-1\}$
 such that each component of $L'$
 contains exactly one point of $\{P_{M_*}(0,Q_j) \st j \in S\}$.
By $H_\cR:\RR \ra \RR$, we denote the holonomy map
 of the foliation $\hat{\cR}^1$
 (see subsection \ref{sec:Reeb} for the definition of $\hat{\cR}^1$).
Proposition \ref{prop:fol} implies that
 there exist $n \geq 1$, $\delta>0$,
 and $(\cG^i)_{i=1}^3 \in \TFol(A_*,n)$
 which satisfies the following properties:
\begin{itemize}
 \item $P_{M_*}^{-1}(L')$ is tangent to $\cG^1 \cap \cG^2$,
 \item $\Theta_1((\cG^i)_{i=1}^3,j)=m+\theta(1)$
 for any $j=0,\cdots,n-1$, and
 \item  both $H_x^\delta((\cG^i)_{i=1}^3,j)$
 and $H_y^\delta((\cG^i)_{i=1}^3,j)$ are
 conjugate to $H_\cR|_{[-2,2]}$ for $j \in S_L$
 and the identity map otherwise.
\end{itemize}
The total foliation $(\cG^i)_{i=1}^3$
 induces a total foliation $(\cF_*^i=P_{M_*}(\cG^i))_{i=1}^3$ on $M_*$.
For each component $K'$ of $L'$,
 the first and the last condition in the above implies
 that there exists an embedding $\psi_{K'}:Z \ra M_*$
 such that $\psi_{K'}(S^1 \times \{(0,0)\})=K'$
 and $\psi_{K'}(\cR^i)=\cG^i|_{\psi_{K'}(Z)}$.
Since $\cG^3$ is transverse to $(\cG^1 \cap \cG^2)$,
 $\cF_*^3|_{\psi_{K'}(Z)}$ is diffeomorphic
 to the product foliation $\{t \times D^2 \st t \in S^1\}$.
Hence, a turbularization of $\cF_*^3$ at $\psi_{K'}(Z)$
 produces an $\cR$-component whose core is $K'$.

By the second condition in the above,
 the framing on $L'$ determined by $\cF_*^1$
 coincides with the one represented by $v_m$.
Hence, we can obtain the required total foliation
 by a turbularization along a tubular neighborhood of $L'$.
\end{proof}

Let $K^\tre$ be the right-handed trefoil on $S^3$.
It is known that $K^\tre$ is a fibered knot with monodromy matrix $A_*$
 (see {\it e.g.} \cite[Section 10.I]{Ro}).
Hence, there exists a diffeomorphism
 $\varphi$ from $M_* \setminus K_0$ to $S^3\setminus K^\tre$,
 an embedding $\psi_{K^\tre}$ from $Z$ to $S^3$,
 and an integer $m_*$
 such that $\psi_{K^\tre}(S^1 \times \{(0,0)\})=K^\tre$ and
\begin{equation}
\label{eqn:varphi}
 \varphi \circ \psi_{K_0}(t,r\cos(2\pi \theta),r\sin(2\pi \theta)))
 =\psi_{K^\tre}(\theta+m_*t,r\cos(2\pi t),r\sin(2\pi t)).
\end{equation}
 for any $t,\theta \in S^1$ and $r \in [0,1]$.
Remark that $\psi_{K^\tre}$ is a $0$-framed null-homotopic embedding
 since $\psi_{K^\tre}(S^1 \times \{(1,0)\})$ is contained in 
 a Seifert surface
 $\cl{\varphi(P_{M_*}(0 \times \{\TT^2 \setminus Q_0\})}$ of $K^\tre$.
We define another embedding $\psi_0$ from $Z$ to $S^3$ by
\begin{displaymath}
 \psi_0(t,x,y)=
 \psi_{K^\tre}\left(\frac{x}{4},
 \frac{y+2}{4}\cos(2\pi t), \frac{y+2}{4}\sin(2\pi t) \right).
\end{displaymath}
Then, the core $\psi_0(S^1 \times \{(0,0)\})$
 bounds a disk $D_0=\psi_{K^\tre}(0 \times D^2(1/2)\}$.
In particular, 
  $\psi_0(S^1 \times \{(0,0)\})$ is a meridian of $K^\tre$.
Since $\psi_0(S^1 \times \{(0,-1)\})$ is contained in $D_0$,
  $\psi_0$ is a $0$-framed unknotted embedding.
See Figure \ref{fig:trefoil}.
\begin{figure}[ht]
\begin{center}
\includegraphics[scale=0.8]{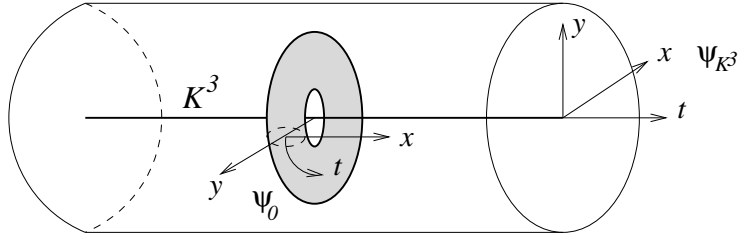}
\end{center}
\caption{The $0$-framed unknotted embedding $\psi_0$}
\label{fig:trefoil}
\end{figure}

\begin{lemma}
\label{lemma:psi0}
For any $\psi_0$-braided link $L$,
 $L'=\varphi^{-1}(L)$ is $\psi_{K_0}$-braided
 and $\varphi$ maps the $(\psi_0,n-m_*)$-framing of $L'$
  to the $(\psi_{K_0},n)$-framing of $L$.
\end{lemma}
\begin{proof}
By direct calculation, we have
\begin{equation*}
\psi_0(t,x,y)
  = \varphi \circ \psi_{K_0}\left(t,
  \frac{y+2}{4}\cos 2\pi\left(\frac{x}{4}-m_*t\right),
  \frac{y+2}{4}\sin 2\pi\left(\frac{x}{4}-m_*t\right)
 \right)
\end{equation*}
 for $(t,x,y) \in Z$.
It implies $\varphi^{-1} \circ \psi_0(t \times D^2)
 \subset \psi_{K_0}(t \times D^2)$ for any $t \in S^1$.
Hence, $L'=\varphi^{-1}(L)$ is $\psi_{K_0}$-braided
 for any $\psi_0$-braided link $L$.
The above equation also implies that
 the $\varphi^{-1}$ maps the $(\psi_0,0)$-framing of $L$
 to the $(\psi_{K_0},-m_*)$-framing of $L'$.
Since the map $\psi_{K_0}^{-1} \circ \varphi^{-1} \circ \psi_0$
 preserves the orientation of $t \times D^2$,
 it implies that the $\varphi^{-1}$ maps the $(\psi_0,n)$-framing of $L$
 to the $(\psi_{K_0},n-m_*)$-framing of $L'$.
\end{proof}

\begin{prop}
\label{prop:any braid}
For any $\psi_0$-braided link $L$,
 there exists a total foliation $(\cF^i)_{i=1}^3$ on $S^3$
 such that any connected component $K$ of $L$ is the core of
 an $\omega(K;\psi_0))+n(K;\psi_0))$-framed $\cR$-component.
\end{prop}
\begin{proof}
The link $L_*=K_0 \cup \varphi^{-1}(L)$ in $M_*$
 is $\psi_{K_0}$-braided.
By Proposition \ref{prop:fol in W*},
 there exists a total foliation $(\cF_*^i)_{i=1}^3$ on $M_*$
 such that each component of $L_*$ is the core of an $\cR$-component
 and the framing determined by $\cF_*^1$ coincides
 with the $(\psi_{K_0},1-m_*)$-framing.
Let $R_0$ be the $\cR$ component of $(\cF_*^i)_{i=1}^3$
 whose core is $K_0$.
Without loss of generality,
 we may assume that $R_0=\psi_{K_0}(S^1 \times D^2(\epsilon))$
 for some $\epsilon>0$.
Then, $a_\cR(R_0)$ is represented by a curve
\begin{displaymath}
 C_0=\psi_{K_0}(\{(t,\epsilon\cos(2\pi(1-m_*)t),
  \epsilon\sin(2\pi(1-m_*)t)) \st t \in S^1\}) 
\end{displaymath}
 with a suitable orientation.

Put $\cF^i=\varphi(\cF_*^i|_{\cl{M_* \setminus R}})$ for $i=1,2,3$
 and $R_0'=\psi_{K^\tre}(S^1 \times D^2(\epsilon))$.
Then, $(\cF^i)_{i=1}^3$ is a total foliation on
 $\cl{S^3 \setminus R_0'}$ with an $\cR$-boundary $\del R_0'$.
By (\ref{eqn:varphi}), we have
 \begin{displaymath}
 \varphi(C_0)=
 \psi_{K^\tre}(\{(t,\epsilon \cos(2\pi t),\epsilon \sin(2\pi t))\}).
\end{displaymath}
Since $\psi_{K^\tre}$ is a $0$-framed embedding,
 we can extend $(\cF^i)_{i=1}^3$ so that
 $R'_0$ is a $(+1)$-framed $\cR$-component with $C(R_0)=K^\tre$.
Since the framing on $L_*$ determined by $\cF_*^1$ coincides
 with the $(\psi_{K_0},1-m_*)$-framing,
 Lemma \ref{lemma:psi0} implies that
 the framing on $L$ determined by $\cF^1$ is the $(\psi_0,1)$-framing.
Since $\psi_0$ is a $0$-framed unknotted embedding,
 it gives the $(\omega(K;\psi_0)+n(K;\psi_0))$-framing
 on each component $K$ of $L$ by Lemma \ref{lemma:frame of braid}.
In particular, each component $K$ of $L$
 is the core of an $(\omega(K;\psi_0)+n(K;\psi_0))$-framed
 $\cR$-component.
\end{proof}

\subsection{A proof of Hardorp's theorem}\label{sec:Hardorp}
First, we show that we can change the framing of an $\cR$-component
 by an arbitrary even integer.
\begin{lemma}
\label{lemma:framing change}
Suppose that a total foliation $(\cF^i)_{i=1}^3$ on $S^3$
 admits a $k$-framed $\cR$-component $R$.
Then, for any integer $n$,
 there exists a total foliation $(\cF_n^i)_{i=1}^3$ on $S^3$
 such that it admits a $(k+2n)$-framed $\cR$-component $R'$
 with $C(R')=C(R)$
 and $\cF_n^i|_{\cl{S^3 \setminus R}}=\cF^i|_{\cl{S^3 \setminus R}}$
 for $i=1,2,3$.
\end{lemma}
\begin{proof}
Let $\psi$ be a $0$-framed unknotted embedding of $Z$ into $S^3$
 and $L=K_1 \cup K_2$ be the $\psi$-braided link
 in Figure \ref{fig:add2}.
\begin{figure}[ht]
\begin{center}
\includegraphics[scale=1.0]{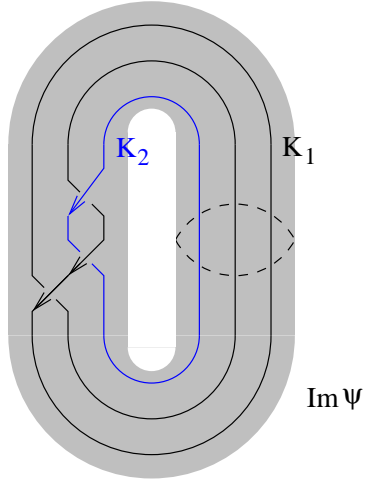}
\end{center}
\caption{The link $L$ for the proof of Lemma \ref{lemma:framing change}}
\label{fig:add2}
\end{figure}
By Proposition \ref{prop:any braid},
 we can take a total foliation $(\cG^i)_{i=1}^3$ on $S^3$
 which admits $\cR$-components $R_1$ and $R_2$ such that
 $K_i$ is the core of $R_i$ for $i=1,2$
 and the framings of $R_1$ and $R_2$ are $+3$ and $+1$, respectively.
Put $M_+ = \cl{S^3 \setminus R_2}$.
Since $L$ is a positive Hopf link, $M_+$ is diffeomorphic to $Z$.
Hence, there exists a diffeomorphism
 $\varphi_+:Z \ra M_+$
 such that $K_1=\varphi_+(S^1 \times \{(0,0)\})$ coincides with
 the core of $R_1$ as oriented knots
 and $(\varphi_+|_{\del Z})_*(a_t)=a_\cR(\del R_2)$,
 where $a_t$ is the homology class in $H_1(\del Z,\ZZ)$
 represented by a map $t \mapsto (t,1,0)$.
It is easy to see that $K_1$ is a $\varphi_+$-braided knot.
Since $K_1$ is $(+3)$-framed and $K_2$ is $(+1)$-framed,
 $\cG_*^1$ gives the $(\varphi_+,2)$-framing of $K_1$.

The lemma for $n=0$ is trivial.
First, we show the lemma for $n=1$.
Let $\psi:Z \ra R$ be a diffeomorphism
 such that $\psi(S^1\times \{(0,0)\})=C(R)$ as oriented knots
 and $\psi_*(a_t)=a_\cR(\del R)$.
By Proposition \ref{prop:gluing},
 if we choose $\varphi_+$ suitably in its isotopy class,
 then we can obtain a total foliation $(\cF_1^i)_{i=1}^3$ on $S^3$
 such that $\cF_1^i|_{R}=\psi \circ \varphi_+^{-1}(\cG_*^i|_{M_+})$
 and  $\cF_1^i|_{\cl{S^3 \setminus R}}=\cF^i|_{\cl{S^3 \setminus R}}$
 for $i=1,2,3$.
Since $\psi$ is a $k$-framed embedding
 and $\cG_*^1$ gives the $(\varphi_+,2)$-framing of $C(R_1)=K_1$,
 $\psi \circ \varphi_+^{-1}(R_1)$ is a $(k+2)$-framed
 $\cR$-component with $C(\psi \circ \varphi_+^{-1}(R_1))=C(R)$.
By inductive construction, it gives the proof
 for $n \geq 1$.

Since $M_-=\cl{S^3 \setminus R_1}$ is diffeomorphic to $Z$,
 we can take a diffeomorphism $\varphi_-:M_+ \ra Z$
 such that $\varphi_-(K_2)= S^1 \times \{(0,0)\}$ as oriented knots
 and $(\varphi_-)_*(a_\cR(\del R_2))=a_t$.
Similar to $\varphi_+$, $K_2$ is $\psi_-$-braided and
 $\cG_*^1$ gives the $(\varphi_-,-2)$-framing of $K_2$.
Hence, the same construction to the above
 completes the proof for $n \leq -1$.
\end{proof}

Now, we give an alternative proof of Hardorp's theorem
 \cite{Ha} with some extension.
\begin{thm}
\label{thm:Hardorp}
For any given closed three-dimensional manifold $M$
 equipped with a spin structure $s$,
 there exists a total foliation $(\cF^i)_{i=1}^3$
 such that
\begin{itemize}
 \item $s$ is the spin structure given by $(\cF^i)_{i=1}^3$,
 \item $(\cF^i)_{i=1}^3$ admits
 two unknotted $\cR$-components $R_+$ and $R_-$,
 \item $R_+$ is $(+1)$-framed and $R_-$ is $(-1)$-framed, and
 \item $R_+$ and $R_-$ are contained
 in mutually disjoint three-dimensional balls.
\end{itemize}
\end{thm}
\begin{proof}
Let $X$ be a four-dimensional $2$-handlebody
 such that $\del X=M$
 and the restriction of the unique spin structure on $X$ to $M$ is $s$.
Let $L_0$ be the Kirby diagram of $X$.
We denote by $n(K)$ the integer-valued framing of
 each component $K$ of $L_0$.
Remark that all $n(K)$ are even since $X$ admits a spin structure.
Take two unknots $K_-$ and $K_+$ which are contained in
 mutually disjoint three-dimensional ball in $S^3 \setminus L_0$.
Put $n(K_-)=0$ and $n(K_+)=2$.

Fix an unknotted embedding $\psi_0$ of the solid torus.
Recall that any link can be $\psi_0$-braided
 by Alexander's theorem (see {\it e.g.} \cite[Theorem 2.1]{Bi2}).
By Proposition \ref{prop:any braid},
 there exists a total foliation $(\cF_0^i)_{i=1}^3$ on $S^3$
 such that each component $K$ of $L_0 \cup K_- \cup K_+$
 is contained in an $\cR$-component with
 $(\omega(K;\psi_0)+n(K;\psi_0))$-framing.
Lemma \ref{lemma:odd framing} implies that these $\cR$-components
 are odd-framed.
By Lemma \ref{lemma:framing change},
 we can modify $(\cF_0^i)_{i=1}^3$ so that
 the framing of an $\cR$-component $R$ is $n(C(R))-1$.
Then, the standard surgery on $L_0$ (not $L_0 \cup K_- \cup K_+$)
 produces a total foliation $(\cF^i)_{i=1}^3$ on $M$.
It is easy to see that each $K_\pm \subset M$
 is the core of a $(\pm 1)$-framed unknotted $\cR$-component
 of $(\cF^i)_{i=1}^3$.
Proposition \ref{prop:surgery formula} implies
 that the spin structure given by $(\cF^i)_{i=1}^3$ is $s$.
\end{proof}

\paragraph{Remark}
The last sentence of Paragraph 14 of Chapter 7 (p.71) of \cite{Ha}
 seems incorrect.
In fact, branched double covering along the unknot
 changes the framing of braided knots in general.
\begin{figure}[ht]
\begin{center}
\includegraphics[scale=1.0]{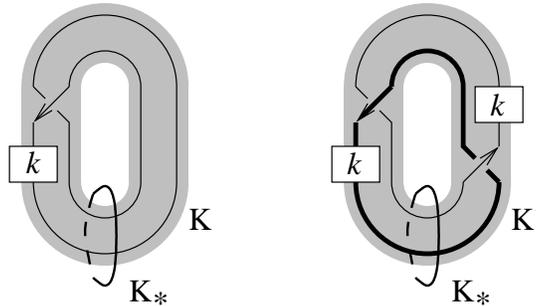}
\end{center}
\caption{Double covering of a solid torus}
\label{fig:covering}
\end{figure}
For example, Figure \ref{fig:covering}
 illustrates a branched double covering along the unknot $K_*$.
The box represents a tangle
 where the difference of the numbers
 of positive and negative crossings is $k$.
Suppose that the knot $K$ in the left-side of the figure
 has the blackboard framing,
 which is equal to the $(k+1)$-framing.
In the right-side of the figure,
 which is a double covering of the left-side,
 the lift $K'$ of the framed knot $K$ has the blackboard framing,
 which is equal to the $k$-framing.
Hence, the knot $K$ is isotopic to $K'$ as a knot,
 but is not isotopic to $K'$ as a framed knot.
It is because one positive crossing in the left-side
 is not counted in the right-side.
The same phenomenon occurs in the setting in Chapter 7 of \cite{Ha}.

\subsection{Proof of Theorem \ref{thm:total}}
\label{sec:proof}
First, we construct a suitable total foliation on $S^3$.
Let $(\cR_+^i)_{i=1}^3$ be the positive total Reeb foliation on $S^3$,
 that is, a total foliation consisting of two $(-1)$-framed
 unknotted $\cR$-components.
\begin{lemma}
\label{lemma:change Hopf}
For any integer $n$,
 there exists a total foliation $(\cG_n^i)_{i=1}^3$ on $S^3$
 with unknotted $\cR$-components $R_+$ and $R_-$
 such that $H((\cG_n^i)_{i=1}^3,(\cR_+^i)_{i=1}^3)=n$,
 $R_+$ is $(+1)$-framed, $R_-$ is $(-1)$-framed,
 and $R_+$ and $R_-$ are contained in mutually disjoint
 three-dimensional balls.
\end{lemma}
\begin{proof}
First, we show the lemma for $n=-1$.
By Theorem \ref{thm:Hardorp},
 there exists a total foliation $(\cF_0^i)_{i=1}^3$ on $S^3$
 with  unknotted $\cR$-components $R_+$ and $R_-$
 such that $R_+$ is $(+1)$-framed, $R_-$ is $(-1)$-framed,
 and they are contained in mutually disjoint three-dimensional balls
 $B_+$ and $B_-$, respectively.
Since $\cl{S^3 \setminus R_+}$ is an unknotted solid torus,
 we can take an {\it orientation reversing} diffeomorphism
 $\varphi$ on $S^3$ so that
 $\varphi(\cl{S^3 \setminus R_+})=R_+$.
Then, $\cl{S^3 \setminus R_+}=\varphi(R_+)$ is a $(-1)$-framed unknotted
 $\cR$-component of $(\varphi(\cF_0^i))_{i=1}^3$.

Let $(\cG_{-1}^i)_{i=1}^3$ be the total foliation
 obtained by gluing $(\cF_0^i)_{i=1}^3$ and $(\varphi(\cF_0^i))_{i=1}^3$
 along $\cR$-components $R_+$ and $\varphi(R_+)$
 as in Subsection \ref{sec:gluing}.
It admits unknotted $\cR$-components $R_-$ and $\varphi(R_-)$
 which are contained in mutually disjoint balls $B_-$
 and $\varphi(B_-) \subset \varphi(\cl{S^3 \setminus R_+})=R_+$.
Then, we have
\begin{eqnarray*}
 H((\cG_{-1}^i)_{i=1}^3,(\cR_+^i)_{i=1}^3) &=&
 H((\cG_{-1}^i)_{i=1}^3,(\cF^i)_{i=1}^3)
 +H((\cF^i)_{i=1}^3,(\cR_+^i)_{i=1}^3)\\
 & = &
 H((\varphi(\cF^i))_{i=1}^3,(\cR_+^i)_{i=1}^3)
 + H((\cF^i)_{i=1}^3,(\cR_+^i)_{i=1}^3)\\
 & = & -1,
\end{eqnarray*}
 where each equality follows from
 the formulas (\ref{eqn:Hopf}), (\ref{eqn:inversion formula}),
 and (\ref{eqn:gluing formula})
 in Subsections \ref{sec:spin} and \ref{sec:surgery}.
Since $R_-$ and $\varphi(R_-)$ have the $(-1)$ and $(+1)$-framings
 respectively, the proof for $n=-1$ is completed.

Second, we show the lemma for $n \leq -1$ by induction.
Suppose that there exists a total foliation $(\cG_n^i)_{i=1}^3$
 which satisfies the assertion of the lemma for some $n \leq -1$.
Let $R'_+$ and $R'_-$ be $(+1)$ and $(-1)$-framed
 unknotted $\cR$-components contained in mutually disjoint balls.
We construct the total foliation $(\cG_{n-1}^i)_{i=1}^3$
 by gluing $(\cG_n^i)_{i=1}^3$ and $(\cG_{-1}^i)_{i=1}^3$
 along $\cR$-components $R'_+$ and $R_-$.
By the formula (\ref{eqn:gluing formula}),
 we have $H((\cG_{n-1}^i)_{i=1}^3,(\cR_+^i)_{i=1}^3)=n-1$.
It is easy to see that
 $\cR$-components $R'_-$ and $\varphi(R_-)$ satisfy the assertion
 of the lemma.

For $n \geq 0$, 
 put $(\cG_n^i)_{i=1}^3=(\varphi(\cG_{-n-1}^i)_{i=1}^3)$.
By the formula (\ref{eqn:inversion formula}),
 we have $H((\cG_n^i)_{i=1}^3,(\cR_+^i)_{i=1}^3)=n$.
It is easy to see that $(\cG_n^i)_{i=1}^3$ is the required one.
\end{proof}

Now, we show the following theorem
 which implies Theorem \ref{thm:total}.
\begin{thm}
\label{thm:total*}
Let $M$ be a closed oriented three-dimensional manifold.
Any homotopy class of total plane fields on $M$
 can be realized by a total foliation $(\cF^i)_{i=1}^3$
 with $(+1)$- and $(-1)$-framed unknotted $\cR$-components.
\end{thm}
\begin{proof}
Fix a spin structure $s$ on $M$.
By Theorem \ref{thm:Hardorp},
 we can take a total foliation $(\cF^i)_{i=1}^3$ on $M$
 such that it admits $(+1)$ and $(-1)$-framed unknotted $\cR$-components
 $R_+$ and $R_-$,
 and the spin structure given by $(\cF^i)_{i=1}^3$ is $s$.
By Proposition \ref{prop:framing},
 it is sufficient to show that there exists
 a total foliation $(\cF_n^i)_{i=1}^3$ on $M$
 such that it admits $(+1)$- and $(-1)$-unknotted $\cR$-components,
 the spin structure given by $(\cF_n^i)_{i=1}^3$ is $s$, and
 $H((\cF_n^i)_{i=1}^3,(\cF^i)_{i=1}^3)=n$ for any given integer $n$.

Take an integer $n$.
Let $(\cG_n^i)_{i=1}^3$ be the total foliation on $S^3$
 obtained in Lemma \ref{lemma:change Hopf} for $n$.
It admits $(+1)$-and $(-1)$-framed $\cR$-components $R'_+$ and $R'_-$
 which are contained in mutually disjoint balls.
Let $(\cF_n^i)_{i=1}^3$ be the total foliation
 obtained by gluing $(\cF^i)_{i=1}^3$ and $(\cG_n^i)_{i=1}^3$
 along $\cR$-components $R_+$ and $R'_-$. 
Since $R_+$ is contained in a three-dimensional ball,
 $(\cF_n^i)_{i=1}^3$ and $(\cF^i)_{i=1}^3$
 give the same spin structure.
By Proposition \ref{prop:gluing formula},
 we obtain $H((\cF_n^i)_{i=1}^3,(\cF^i)_{i=1}^3)=n$.
It is easy to see that $R_-$ and $R'_+$ are
 $(-1)$ and $(+1)$-framed unknotted $\cR$-components
 of $(\cF_n^i)_{i=1}^3$.
\end{proof}

\section{Bi-contact structures}
\label{chap:contact}

First, we recall some basic definitions
 and results on contact topology.
A plane field $\xi$ on a three-dimensional manifold $M$
 is called a {\em positive} (resp. {\em negative}) {\em contact structure}
 if it is the kernel of a $1$-form $\alpha$ with
 $\alpha \wedge d\alpha>0$ (resp. $\alpha \wedge d\alpha < 0$).
We say a knot $K$ in $(M,\xi)$ is {\it Legendrian}
 if it is tangent to $\xi$.
{\it The Thurston-Bennequin invariant} $tb(K,\xi)$
 is the integer-valued framing of $K$ given by $\xi$.
The {\it rotation} $rot(K,\xi)$ is the Euler number
 $\chi(\xi,\Sigma,K)$ of $\xi$ on a Seifert surface $\Sigma$
 relative to $K$.

A contact structure $\xi$ on $M$ is called {\it overtwisted}
 if there exists a Legendrian unknot $K$ such that $tb(K,\xi)=0$.
We say $\xi$ is {\it tight} if it is not overtwisted.
It is known that if $\xi$ is tight, then 
 any null-homologous Legendrian knot $K$ satisfies
 {\it the Thurston-Bennequin inequality} :
\begin{displaymath}
tb(K,\xi)+\chi(\Sigma) \leq -|rot(K,\xi)|, 
\end{displaymath}
 where $\chi(\Sigma)$ is the Euler number
 of a Seifert surface $\Sigma$ of $K$.

\begin{thm}[Eliashberg,\cite{El}]
\label{thm:Eliashberg}
Let $M$ be a three-dimensional closed manifold,
any homotopy class of plane fields on $M$ contains exactly one positive
 ({\it resp.} negative) overtwisted contact structure up to isotopy.
\end{thm}

The following lemma gives a criterion for the overtwistedness
 of a contact structure which is close to a foliation of a total foliation.
\begin{lemma}
\label{lemma:overtwisted}
Let $(\cF^i)_{i=1}^3$ be a total foliation
 on a three-dimensional manifold $M$
 and suppose it admits a $(+1)$-framed ({\it resp.} $(-1)$-framed)
 unknotted $\cR$-component $R$.
Then, any positive ({\it resp.} negative) contact structure
 which is sufficiently $C^0$-close to $T\cF^1$ is overtwisted.
\end{lemma}
\begin{proof}
We show the assertion for positive contact structures.
The proof for negative contact structures
 is obtained by reversing the orientation.

The foliation $\cF^1|_{\del R}$ admits a closed leaf $C$
 which is isotopic to the core of $R$ as an oriented knot in $M$.
In particular, $C$ is unknotted.
The foliation $\cF^1$ gives the $(+1)$-framing on $C$.

Recall that $\del R$ is a leaf of $\cF^3$.
By Lemma \ref{lemma:perturbation},
 if a smooth plane field $\xi$ is sufficiently $C^0$-close to $T\cF^1$,
 there exists a closed curve $C_\xi$ in $\del R$
 which is tangent to $\xi \cap T\cF^3$ and isotopic to $C$ in $\del R$.
The curve $C_\xi$ is unknotted in $M$, and hence,
 it bounds a disk $D_\xi$.
Since $\xi \cap T\cF^3$ gives an trivialization of $\xi$ on $D_\xi$,
 we have $rot(C_\xi,\xi)=0$.

By the transversality,
 $\cF^1$ and $\cF^3$ define the same framing on $C$,
 and $\xi$ and $\cF^3$ define the same framing on $C_\xi$.
Hence, the framing on $C_\xi$ given $\xi$ is $(+1)$.
In particular, $tb(C_\xi,\xi)=+1$.
It violates the Thurston-Bennequin inequality since
\begin{displaymath}
tb(C_\xi,\xi)+\chi(D_\xi)=2>0=|rot(C_\xi,\xi)|.
\end{displaymath}
\end{proof}

Now, we prove Theorem \ref{thm:bi-contact}.
Let $M$ be a closed and oriented three-dimensional manifold.
Fix a pair $(\xi,\eta)$ of positive and negative contact structures
 such that they are homotopic as plane fields
 and their Euler class is zero.
Then, there exists a total plane field $(\xi^i)_{i=1}^3$ on $M$
 such that $\xi^i$ is homotopic to $\xi$ and $\eta$ for $i=1,2,3$.
By Theorem \ref{thm:total*},
 $(\xi_i)_{i=1}^3$ is homotopic to
 a total foliation $(\cF^i)_{i=1}^3$ on $M$
 which admits $(+1)$ and $(-1)$-framed unknotted $\cR$-components.

By the fundamental theorem of confoliations
 \cite[Theorem 2.4.1]{ET},
 we can take a bi-contact structure $(\xi_*,\eta_*)$ on $M$
 so that $\xi_*$ is $C^0$-sufficiently close to $\cF^1$
 and $\eta_*$ is $C^0$-sufficiently close to $\cF^2$.
Lemma \ref{lemma:overtwisted} implies that
 both $\xi_*$ and $\eta_*$ are overtwisted.
By Theorem \ref{thm:Eliashberg},
 $\xi_*$ and $\eta_*$ are isotopic to $\xi$ and $\eta$
 as contact structures, respectively.

\bibliographystyle{plain}

\begin{thebibliography}{10}

\bibitem{Bi}
Joan~S. Birman.
\newblock On braid groups.
\newblock {\em Comm. Pure Appl. Math.}, 22:41--72, 1969.

\bibitem{Bi2}
Joan~S. Birman.
\newblock {\em Braids, links, and mapping class groups}.
\newblock Princeton University Press, Princeton, N.J., 1974.
\newblock Annals of Mathematics Studies, No. 82.

\bibitem{CC}
Alberto Candel and Lawrence Conlon.
\newblock {\em Foliations. {II}}, volume~60 of {\em Graduate Studies in
  Mathematics}.
\newblock American Mathematical Society, Providence, RI, 2003.

\bibitem{Du2}
Emmanuel Dufraine.
\newblock About homotopy classes of non-singular vector fields on the
  three-sphere.
\newblock {\em Qual. Theory Dyn. Syst.}, 3(2):361--376, 2002.

\bibitem{Du}
Emmanuel Dufraine.
\newblock Classes d'homotopie de champs de vecteurs {M}orse-{S}male sans
  singularit\'e sur les fibr\'es de {S}eifert.
\newblock {\em Enseign. Math. (2)}, 51(1-2):3--30, 2005.

\bibitem{El}
Yakov~M. Eliashberg.
\newblock Classification of overtwisted contact structures on {$3$}-manifolds.
\newblock {\em Invent. Math.}, 98(3):623--637, 1989.

\bibitem{ET}
Yakov~M. Eliashberg and William~P. Thurston.
\newblock {\em Confoliations}, volume~13 of {\em University Lecture Series}.
\newblock American Mathematical Society, Providence, RI, 1998.

\bibitem{GS}
Robert~E. Gompf and Andr{\'a}s~I. Stipsicz.
\newblock {\em {$4$}-manifolds and {K}irby calculus}, volume~20 of {\em
  Graduate Studies in Mathematics}.
\newblock American Mathematical Society, Providence, RI, 1999.

\bibitem{Go}
Juan Gonz{\'a}lez-Meneses.
\newblock New presentations of surface braid groups.
\newblock {\em J. Knot Theory Ramifications}, 10(3):431--451, 2001.

\bibitem{Ha}
Detlef Hardorp.
\newblock All compact orientable three dimensional manifolds admit total
  foliations.
\newblock {\em Mem. Amer. Math. Soc.}, 26(233):vi+74, 1980.

\bibitem{Mil}
John Milnor.
\newblock Spin structures on manifolds.
\newblock {\em Enseignement Math. (2)}, 9:198--203, 1963.

\bibitem{Mi}
Yoshihiko Mitsumatsu.
\newblock Anosov flows and non-{S}tein symplectic manifolds.
\newblock {\em Ann. Inst. Fourier (Grenoble)}, 45(5):1407--1421, 1995.

\bibitem{Mi2}
Yoshihiko Mitsumatsu.
\newblock Foliations and contact structures on 3-manifolds.
\newblock In {\em Foliations: geometry and dynamics (Warsaw, 2000)}, pages
  75--125. World Sci. Publ., River Edge, NJ, 2002.

\bibitem{Ro}
Dale Rolfsen.
\newblock {\em Knots and links}, volume~7 of {\em Mathematics Lecture Series}.
\newblock Publish or Perish Inc., Houston, TX, 1990.
\newblock Corrected reprint of the 1976 original.

\bibitem{TS}
Itiro Tamura and Atsushi Sato.
\newblock On transverse foliations.
\newblock {\em Inst. Hautes \'Etudes Sci. Publ. Math.}, (54):205--235, 1981.

\bibitem{Wo}
John~W. Wood.
\newblock Foliations on {$3$}-manifolds.
\newblock {\em Ann. of Math. (2)}, 89:336--358, 1969.

\end{thebibliography}

\def\cprime{$'$}

\end{document}